\documentclass[a4paper, 10pt,reqno]{amsart}

\usepackage{amsthm}
\usepackage{amsmath}
\usepackage{amsfonts}
\usepackage{amssymb}
\usepackage{latexsym}

\usepackage{graphicx}

\usepackage{tabularx}
\usepackage[margin=1.2in]{geometry}
\usepackage{enumitem}
\usepackage{ae}
\usepackage[T1]{fontenc}

\newtheorem{teo}{Theorem}
\newtheorem{cor}{Corollary}

\newtheorem{rem}{Remark}
\newtheorem{prop}{Proposition}
\newtheorem{lem}{Lemma}

\newcommand{\R}{\mathbb{R}}

\begin{document}
\title[A nonexist. result for sign-changing sol. of the B-N problem in low dimensions]{A nonexistence result for sign-changing solutions of the Brezis-Nirenberg problem in low dimensions}
\author{Alessandro Iacopetti and Filomena Pacella}
\date{}
\thanks{Research partially supported by MIUR-PRIN project-201274FYK7\underline\ 005 and GNAMPA-INDAM.}
\address[Alessandro Iacopetti]{Dipartimento di Matematica e Fisica, Universit\`a degli Studi di Roma Tre, L.go S. Leonardo Murialdo 1, 00146 Roma, Italy}
\email{iacopetti@mat.uniroma3.it}
\address[Filomena Pacella]{Dipartimento di Matematica, Universit\`a di Roma "La Sapienza", P.le Aldo Moro 5, 00185 Roma, Italy}
\email{pacella@mat.uniroma1.it}

\begin{abstract}
We consider the Brezis-Nirenberg problem:
\begin{equation*}
\begin{cases}
-\Delta u = \lambda u + |u|^{2^* -2}u & \hbox{in}\ \Omega\\
u=0 & \hbox{on}\ \partial \Omega,
\end{cases}
\end{equation*}
where $\Omega$ is a smooth bounded domain in $\R^N$, $N\geq 3$, $2^{*}=\frac{2N}{N-2}$ is the critical Sobolev exponent and $\lambda>0$  a positive parameter.

The main result of the paper shows that if $N=4,5,6$ and $\lambda$ is close to zero there are no sign-changing solutions of the form 
$$u_\lambda=PU_{\delta_1,\xi}-PU_{\delta_2,\xi}+w_\lambda, $$
where $PU_{\delta_i}$ is the projection on $H_0^1(\Omega)$ of the regular positive solution of the critical problem in $\R^N$, centered at a point $\xi \in \Omega$ and $w_\lambda$ is a remainder term.

Some additional results on norm estimates of $w_\lambda$ and about the concentrations speeds of tower of bubbles in higher dimensions are also presented.
\end{abstract}
 \maketitle
\section{Introduction}
In this paper we study the semilinear elliptic problem:
\begin{equation}\label{PBN}
\begin{cases}
-\Delta u = \lambda u + |u|^{2^* -2}u & \hbox{in}\ \Omega\\
u=0 & \hbox{on}\ \partial \Omega,
\end{cases}
\end{equation}
where $\Omega$ is a smooth bounded domain in $\R^N$, $N\geq 3$, $\lambda$ is a positive real parameter and $2^{*}=\frac{2N}{N-2}$ is the critical Sobolev exponent for the embedding of $H_0^1(\Omega)$ into $L^{2^*}(\Omega)$.

This problem is known as "the Brezis-Nirenberg problem" because the first fundamental results about the existence of positive solutions were obtained by H. Brezis and L. Nirenberg in 1983 in the celebrated paper \cite{BN}. From their results it came out that the dimension was going to play a crucial role in the study of \eqref{PBN}. Indeed they proved that if $N\geq 4$ there exists a positive solution of \eqref{PBN} for every $\lambda \in (0,\lambda_1(\Omega))$, $\lambda_1(\Omega)$ being the first eigenvalue of $-\Delta$ in $\Omega$ with Dirichlet boundary conditions, while if $N=3$ positive solutions exists only for $\lambda$ away from zero. In particular, in the case of the ball $B$ they showed that there are no positive solutions in the interval $(0,\frac{\lambda_1(B)}{4})$. 

Since then several other interesting results were obtained for positive solutions, in particular about the asymptotic behavior of solutions, mainly for $N\geq 5$ because also the case $N=4$ presents more difficulties compared to the higher dimensional ones.

Concerning the case of sign-changing solutions, existence results hold if $N\geq4$ both for $\lambda \in (0,\lambda_1(\Omega))$ and $\lambda > \lambda_1(\Omega)$ as shown in \cite{ABP2}, \cite{CW}, \cite{CFP}.

The case $N=3$ presents even more difficulties than in the study of positive solutions. In particular in the case of the ball is not yet known what is the least value $\bar \lambda$ of the parameter $\lambda$ for which sign-changing solutions exist, neither whether $\bar\lambda$ is larger or smaller than $\lambda_1(B)/4$. This question, posed by H. Brezis, has been given a partial answer in \cite{AMP3}. However it is interesting to observe that in the study of sign-changing solutions even the "low dimensions" $N=4,5,6$ exhibit some peculiarities. Indeed it was first proved by Atkinson, Brezis and Peletier in \cite{ABP} that if $\Omega$ is a ball there exists $\lambda^*=\lambda^*(N)$ such that there are no radial sign-changing solutions of \eqref{PBN} for $\lambda \in (0,\lambda^*)$. Later this result was reproved in \cite{AY2} in a different way.

Moreover for $N\geq 7$ a recent result of Schechter and Zou \cite{SZ} shows that in any bounded smooth domain
there exist infinitely many sign-changing solutions for any $\lambda>0$. Instead if $N=4,5,6$ only $N+1$ pairs of solutions, for all $\lambda>0$, have been proved to exist in \cite{CW} but it is not clear that they change sign.

Coming back to the nonexistence result of \cite{ABP} and \cite{AY2} an interesting question would be to see whether and in which way it could be extended to other bounded smooth domains. 

Since the result of \cite{ABP} and \cite{AY2} concerns nodal radial solutions in the ball the first issue is to understand what are, in general bounded domains, the sign-changing solutions which play the same role as the radial nodal solutions in the case of the ball. A main property of a radial nodal solution in the ball is that its nodal set does not touch the boundary therefore, a class of solutions to consider, in general bounded domains, could be the one made of functions which have this property.

Moreover, in analyzing the asymptotic behavior of least energy nodal radial solutions $u_\lambda$ in the ball, as $\lambda \rightarrow 0$, in dimension $N\geq7$ (in which case they exist for all $\lambda \in (0,\lambda_1(B)$), see \cite{CSS}) one can prove (see \cite{Iac1}) that their limit profile is that of a "tower of two bubbles". This terminology means that the positive part and the negative part of the solutions $u_\lambda$ concentrate at the same point (which is obviously the center of the ball) as $\lambda \rightarrow 0$ and each one has the limit profile, after suitable rescaling, of a "standard" bubble in $\R^N$, i.e. of a positive solution of the critical exponent problem in $\R^N$. More precisely the solutions $u_\lambda$ can be written in the following way:
\begin{equation}\label{formasoluz}
u_\lambda=PU_{\delta_1,\xi}-PU_{\delta_2,\xi}+w_\lambda,
\end{equation}
where $PU_{\delta_i,\xi}$, $i=1,2$ is the projection on $H_0^1(\Omega)$ of the regular positive solution of the critical problem in $\R^N$, centered at $\xi=0$, with rescaling parameter $\delta_i$ and $w_\lambda$ is a remainder term which converges to zero in $H_0^1(\Omega)$.

It is also interesting to observe that, thanks to a recent result of \cite{IACVAIR}, sign-changing bubble-tower solutions exist also in bounded smooth symmetric domains in dimension $N\geq 7$ for $\lambda$ close to zero, and they have the property that their nodal set does not touch the boundary of the domain.

In view of all these remarks we are entitled to assert that in general bounded domains sign-changing solutions which behave as the radial ones in the ball, at least for $\lambda$ close to zero, are the ones which are of the form
\eqref{formasoluz}.
Hence a natural extension of the nonexistence result of \cite{ABP} and \cite{AY2} would be to show that, in dimension $N=4,5,6$, sign-changing solutions of the form \eqref{formasoluz} do not exist in any bounded smooth domain.

This is indeed the main aim of this paper. Let us also note that in the 3-dimensional case a similar nonexistence result was already proved in \cite{AMP3}. Indeed, in studying the asymptotic behavior of low-energy nodal solutions it was shown in \cite{AMP3} that their positive and negative part cannot concentrate at the same point, as $\lambda$ tends to a limit value $\bar\lambda>0$. In the case $N\geq 4$ this question was left open in \cite{AMP2}. Therefore our results also complete the analysis made in these last two papers.

To state precisely our result let us recall that the functions

\begin{equation}\label{Udelta}
U_{\delta,\xi}(x)=\alpha_N \frac{\delta^{\frac{N-2}{2}}}{\left(\delta^2+|x-\xi|^2\right)^{\frac{N-2}{2}}},\qquad \delta>0, \ \xi \in \R^N,
\end{equation}
$\alpha_N:=[N(N-2)]^{\frac{N-2}{4}}$, describe all regular positive solutions of the problem
\begin{equation*}
\begin{cases}
-\Delta U = U^{\frac{N+2}{N-2}} & \hbox{in} \ \R^N,\\
U(x) \rightarrow 0, & \hbox{as} \ |x|\rightarrow + \infty.
\end{cases}
\end{equation*}
Then, denoting by $PU_{\delta}$ their projection on $H_0^1(\Omega)$, and by $\|u\|:=\int_\Omega |\nabla u|^2 \ dx$ for any $u \in H_0^1(\Omega)$, we have:

\begin{teo}\label{teoprincipintro}
Let $N=4,5,6$ and $\xi$ a point in the domain $\Omega$. Then, for $\lambda$ close to zero, Problem \eqref{PBN} does not admit any sign-changing solution $u_\lambda$ of the form \eqref{formasoluz} with $\delta_i=\delta_i(\lambda)$, $i=1,2$, such that $\delta_2=o(\delta_1)$, $\|w_\lambda\| \rightarrow 0$ and $|w_\lambda|=o(\delta_1^{-\frac{N-2}{2}})$, $|\nabla w_\lambda|=o(\delta_1^{-\frac{N}{2}})$ uniformly in compact subsets of $\Omega$, as $\lambda \rightarrow 0$.
\end{teo}
The previous notations mean that
$\frac{|w_\lambda|}{\delta_1^{-\frac{N-2}{2}}}$, $\frac{|\nabla w_\lambda|}{\delta_1^{-\frac{N}{2}}}$ converge to zero as $\lambda \rightarrow 0$ uniformly in compact subsets of $\Omega$.

The proof of the above theorem is based on a Pohozaev identity and fine estimates which are derived in a different way in the case $N=4$ or $N=5,6$. We would like to point out that it cannot be deduced by the proof of Theorem 3.1 of \cite{AMP3} which holds only in dimension three.

Concerning the assumption on the $C^1$-norm in compact subsets of $\Omega$  of the remainder term $w_\lambda$, whose gradient is only required not to blow up too fast, in Section 4 we show that it is almost necessary.

Note that we do not even require that $w_\lambda \rightarrow 0$ uniformly in $\Omega$ neither that it remains bounded as $\lambda \rightarrow 0$, but only a control of possible blow-up of $|w_\lambda|$ and $|\nabla w_\lambda|$. We delay to the next sections some further comments and comparisons with the case $N\geq 7$.

Finally in the last section we show that in dimension $N\geq 7$ if $(u_\lambda)$ is a family of solutions of type \eqref{formasoluz} with $|w_\lambda|$, $|\nabla w_\lambda|$ as in Theorem \ref{teoprincipintro} and $\delta_i=d_i \lambda^{\alpha_i}$, for some positive numbers $d_i=d_i(\lambda)$ with $0<c_1<d_i<c_2$, for all sufficiently small $\lambda$, and $0<\alpha_1<\alpha_2$, then necessarily: 
\begin{equation}\label{formaalpha}
\alpha_1=\frac{1}{N-4}, \ \ \ \alpha_2=\frac{3N-10}{(N-4)(N-6)}. 
\end{equation}

In other words we prove that if the concentration speeds are powers of $\lambda$ then necessarily the exponent must be as in \eqref{formaalpha}. Note that these are exactly the type of speeds assumed in \cite{IACVAIR} to construct the tower of bubbles in higher dimensions.
%and $w_\lambda$ is a remainder term such that $\|w_\lambda\|_\Omega \rightarrow 0$ as $\lambda \rightarrow 0$, %and $(w_\lambda,PU_{\delta_i})=\left(w_\lambda,\frac{\partial PU_{\delta_i}}{\partial \delta_i}\right)=0$, for $i=1,2$, for all small $\lambda>0$,
 %where $\|u\|_\Omega:=\int_\Omega|\nabla u|^2 \ dx$. %and $(\cdot,\cdot)$ is its associated scalar product on  $H_0^1(\Omega)$. 
%We assume that $\delta_2=o(\delta_1)$, as $\lambda \rightarrow 0$, so that $PU_{\delta_1,\xi}$, $PU_{\delta_2,\xi}$ concentrates at $\xi$ with different speeds when $\lambda \rightarrow 0$. 
%If $\xi=0$ we write $U_{\delta}$, $PU_{\delta}$ instead of $U_{\delta,0}$, $PU_{\delta,0}$.
%
%Such a kind of solutions have been obtained for $N\geq 7$ by Iacopetti and Vaira in \cite{IACVAIR} in the case of symmetric domains.
%The aim of this paper is to show that, under some additional assumptions on the remainder term, such a kind of solutions cannot exist in dimension $N=4,5,6$.
 %It's easy to see that $\|u_\lambda\|^2 \rightarrow 2 S^{N/2}$ as $\lambda \rightarrow 0$.
\section{Some preliminary results}
\begin{lem}\label{svilprbubb}
Let $\Omega$ be a smooth bounded domain of $\R^N$ and let $(\xi,\delta)\in \Omega\times\R^+$. As $\delta \rightarrow 0$ it holds:
$$PU_{\delta,\xi}(x)=U_{\delta,\xi}(x) -\alpha_N \delta^{\frac{N-2}{2}}H(x,\xi) + o(\delta^{\frac{N-2}{2}}), \ x \in \Omega $$
$C^1$-uniformly on compact subsets of $\Omega$, where $H$ is the regular part of the Green function for the Laplacian. Moreover, setting $\varphi_{\xi,\delta}(x):=U_{\delta,\xi}(x)-PU_{\delta,\xi}(x)$, the following uniform estimates hold:
\begin{description}
\item[(i)] $0\leq \varphi_{\xi,\delta}\leq U_{\delta,\xi}$,
\item[(ii)]  $\|\varphi_{\xi,\delta}\|^2=O\left((\frac{\delta}{d})^{N-2}\right)$,
\end{description}
%$|\varphi_{\xi,\delta}|_{2^*,\Omega}=O\left((\frac{\delta}{d})^{\frac{N-2}{2}}\right)$,
where $d=d(\xi,\partial\Omega)$ is the euclidean distance between $\xi$ and the boundary of $\Omega$.
\end{lem}
\begin{proof}
See \cite{Rey}, Proposition 1 and its proof.
\end{proof}

\begin{lem}\label{lem1cc}
Let $N\geq 4$ and $(u_\lambda)$ be a family of sign-changing solutions of (\ref{PBN}) satisfying
$$\|u_\lambda\|^2 \rightarrow 2 S^{N/2}, \ \ \ \hbox{as}\ \lambda \rightarrow 0. $$
Then, for all sufficiently small $\lambda>0$, the set $\Omega \setminus \{x \in \Omega; \ u_\lambda(x)=0\}$ has exactly two connected components.
\end{lem}
\begin{proof}
Let us consider the nodal set $Z_\lambda:=\{x \in \Omega; \ u_\lambda(x)=0\}$ and let $\Omega_1$ be a connected component of $\Omega \setminus Z_\lambda$. Multiplying (\ref{PBN}) by $u_\lambda$ and integrating on $\Omega_1$, we get that
$$ \int_{\Omega_1} |\nabla u_\lambda|^2 \ dx \geq S^{N/2} (1+o(1)),$$
where we have used the Sobolev embedding and the fact that $\lambda \rightarrow 0$ and $\lambda_1(\Omega_1) \int_{\Omega_1} u_\lambda^2 \ dx \leq \int_{\Omega_1} |\nabla u_\lambda|^2 \ dx$, where $\lambda_1(\Omega_1)$ is the first Dirichlet eigenvalue of $-\Delta$ on $\Omega_1$. 

Since $\|u_\lambda\|^2 \rightarrow 2S^{N/2}$, as $\lambda \rightarrow 0$, then for all sufficiently small $\lambda>0$ we deduce that  $\Omega \setminus Z_\lambda$ can have only two connected components.
\end{proof}

We recall now the Pohozaev identity for solutions of semilinear problems which are not necessarily zero on the boundary. Let $D$ be a bounded domain in $\R^N$, $N\geq3$, with smooth boundary and consider the equation 
\begin{equation}\label{eq1yyl}
-\Delta u = f(u) \ \ \hbox{in} \ D,
\end{equation}
where $s \mapsto f(s)$ is a continuos function. Denoting $F(s):=\int_0^s f(t) \ dt$, we have:\\
\begin{prop}\label{yylpohozaev}
Let $u$ be a $C^2$-solution of (\ref{eq1yyl}), then

\begin{equation}\label{yylp}
\begin{array}{ll}
\displaystyle & \displaystyle \int_D \left\{N F(u) - \frac{N-2}{2} u f(u)\right\} \ dx\\ 
\displaystyle =& \displaystyle \int_{\partial D} \left\{ \sum_{i=1}^N x_i \nu_i \left(F(u)-\frac{1}{2}|\nabla u|^2\right)+ \frac{\partial u}{\partial \nu}\sum_{i=1}^N x_i u_{x_i}+ \frac{N-2}{2}u\frac{\partial u}{\partial \nu}\right\} \ d\sigma,
\end{array}
\end{equation}

where $\nu$ denotes the outer normal to the boundary and $u_{x_i}$ is the partial derivative with respect to $x_i$ of $u$.
\end{prop}

%\begin{proof}
%For the proof see \cite{YYL}.
%\end{proof}

The following lemma gives information on the asymptotic behavior of the nodal set $Z_\lambda$ of solutions of \eqref{PBN} as $\lambda \rightarrow 0$.
 
\begin{lem}
Let $N\geq4$, $\xi \in \Omega$ and let $(u_\lambda)$ be a family of solutions of (\ref{PBN}), such that $u_\lambda=PU_{\delta_1,\xi}-PU_{\delta_2,\xi}+w_\lambda$, with $\delta_1=\delta_1(\lambda)$ and $\delta_2=\delta_2(\lambda)$ satisfying
\begin{equation*}
\delta_2=o(\delta_1) \ \ \hbox{and}\ \ \|w_\lambda\| \rightarrow 0,\ \hbox{as}\ \lambda \rightarrow 0. 
\end{equation*}
Moreover, assume that $w_\lambda$ satisfies $|w_\lambda|=o(\delta_1^{-\frac{N-2}{2}})$ uniformly in compact subsets of $\Omega$. Then, for all small $\epsilon>0$ there exists $\lambda_\epsilon>0$ such that the nodal set  $Z_\lambda$ is contained in the annular region $A_{r_1,r_2}(\xi):=\{x \in \Omega;\ r_1<|x-\xi|<r_2\}$, for all $\lambda \in (0,\lambda_\epsilon)$, where $r_1:=\delta_1^{\frac{1}{2}-\epsilon}\delta_2^{\frac{1}{2}+\epsilon}$, $r_2:=\delta_1^{\frac{1}{2}+\epsilon}\delta_2^{\frac{1}{2}-\epsilon}$.
\end{lem}

\begin{proof}
Without loss of generality we assume that $\xi=0$.
Let us fix a small $\epsilon>0$  and a compact neighborhood of the origin $K$. Thanks to the assumptions and Lemma \ref{svilprbubb}, we have the following expansion $u_\lambda(x)=U_{\delta_1}(x)-U_{\delta_2}(x)+o(\delta_1^{-\frac{N-2}{2}})$, which is uniform with respect to $x \in K$ and to all small $\lambda>0$. By definition, for all sufficiently small $\lambda>0$,  we have that $A_{r_1,r_2}(0) \subset K$. For $x$ such that $|x|=r_1$ we have:
\begin{eqnarray*}
U_{\delta_1}(x)&=&\alpha_N \frac{\delta_1^{\frac{N-2}{2}}}{(\delta_1^2+\delta_1^{1-2\epsilon}\delta_2^{1+2\epsilon})^{\frac{N-2}{2}}}=\alpha_N \frac{\delta_1^{-\frac{N-2}{2}}}{[1+(\frac{\delta_2}{\delta_1})^{1+2\epsilon}]^{\frac{N-2}{2}}}\\
&=&\alpha_N\ {\delta_1^{-\frac{N-2}{2}}}- \alpha_N \frac{N-2}{2}{\delta_1^{-\frac{N-2}{2}}}\left(\frac{\delta_2}{\delta_1}\right)^{1+2\epsilon} + o\left({\delta_1^{-\frac{N-2}{2}}}\left(\frac{\delta_2}{\delta_1}\right)^{1+2\epsilon}\right),
\end{eqnarray*}
and
\begin{eqnarray*}
U_{\delta_2}(x)&=&\alpha_N \frac{\delta_2^{\frac{N-2}{2}}}{(\delta_2^2+\delta_1^{1-2\epsilon}\delta_2^{1+2\epsilon})^{\frac{N-2}{2}}}=\alpha_N \frac{\delta_2^{\frac{N-2}{2}}\delta_1^{-\frac{N-2}{2}+(N-2)\epsilon} \delta_2^{{-\frac{N-2}{2}-(N-2)\epsilon}}}{[1+(\frac{\delta_2}{\delta_1})^{1-2\epsilon}]^{\frac{N-2}{2}}}\\
&=&\alpha_N \frac{{\delta_1^{-\frac{N-2}{2}} \left(\frac{\delta_2}{\delta_1}\right)^{-(N-2)\epsilon}}}{[1+(\frac{\delta_2}{\delta_1})^{1-2\epsilon}]^{\frac{N-2}{2}}}\\
&=&\alpha_N\ {\delta_1^{-\frac{N-2}{2}}}\left(\frac{\delta_2}{\delta_1}\right)^{-(N-2)\epsilon}- \alpha_N \frac{N-2}{2}{\delta_1^{-\frac{N-2}{2}}}\left(\frac{\delta_2}{\delta_1}\right)^{1-N\epsilon} + o\left({\delta_1^{-\frac{N-2}{2}}}\left(\frac{\delta_2}{\delta_1}\right)^{1-N\epsilon}\right).
\end{eqnarray*}
Hence, for $x \in K$, such that $|x|=r_1$, we have $$u_\lambda(x)=\alpha_N\ {\delta_1^{-\frac{N-2}{2}}}\left(1-\left(\frac{\delta_2}{\delta_1}\right)^{-(N-2)\epsilon}\right)+o(\delta_1^{-\frac{N-2}{2}}) <0$$
for all sufficiently small $\lambda>0$. On the other hand, by similar computations (just changing the sign of $\epsilon$ in every term of the previous equations), for $x$ such that $|x|=r_2$ we have
 $$u_\lambda(x)=\alpha_N\ {\delta_1^{-\frac{N-2}{2}}}\left(1-\left(\frac{\delta_2}{\delta_1}\right)^{+(N-2)\epsilon}\right)+o(\delta_1^{-\frac{N-2}{2}}) >0$$
for all sufficiently small $\lambda>0$. %For the sake of completeness we write explicitly the computations:

%\begin{eqnarray*}
%U_{\delta_1}(x)&=&\alpha_N \frac{\delta_1^{\frac{N-2}{2}}}{(\delta_1^2+\delta_1^{1+2\epsilon}\delta_2^{1-2\epsilon})^{\frac{N-2}{2}}}=\alpha_N \frac{\delta_1^{-\frac{N-2}{2}}}{[1+(\frac{\delta_2}{\delta_1})^{1-2\epsilon}]^{\frac{N-2}{2}}}\\
%&=&\alpha_N\ {\delta_1^{-\frac{N-2}{2}}}- \alpha_N \frac{N-2}{2}{\delta_1^{-\frac{N-2}{2}}}\left(\frac{\delta_2}{\delta_1}\right)^{1-2\epsilon} + o\left({\delta_1^{-\frac{N-2}{2}}}\left(\frac{\delta_2}{\delta_1}\right)^{1-2\epsilon}\right),
%\end{eqnarray*}
%
%and
%
%\begin{eqnarray*}
%U_{\delta_2}(x)&=&\alpha_N \frac{\delta_2^{\frac{N-2}{2}}}{(\delta_2^2+\delta_1^{1+2\epsilon}\delta_2^{1-2\epsilon})^{\frac{N-2}{2}}}=\alpha_N \frac{\delta_2^{\frac{N-2}{2}}\delta_1^{-\frac{N-2}{2}-(N-2)\epsilon} \delta_2^{{-\frac{N-2}{2}+(N-2)\epsilon}}}{[1+(\frac{\delta_2}{\delta_1})^{1+2\epsilon}]^{\frac{N-2}{2}}}\\
%&=&\alpha_N \frac{{\delta_1^{-\frac{N-2}{2}} \left(\frac{\delta_2}{\delta_1}\right)^{+(N-2)\epsilon}}}{[1+(\frac{\delta_2}{\delta_1})^{1+2\epsilon}]^{\frac{N-2}{2}}}\\
%&=&\alpha_N\ {\delta_1^{-\frac{N-2}{2}}}\left(\frac{\delta_2}{\delta_1}\right)^{+(N-2)\epsilon}- \alpha_N \frac{N-2}{2}{\delta_1^{-\frac{N-2}{2}}}\left(\frac{\delta_2}{\delta_1}\right)^{1+N\epsilon} + o\left({\delta_1^{-\frac{N-2}{2}}}\left(\frac{\delta_2}{\delta_1}\right)^{1+N\epsilon}\right).
%\end{eqnarray*}
%
From Lemma \ref{lem1cc} and since $u_\lambda$ is a continuos function we deduce that $Z_\lambda \subset A_{r_1,r_2}(0)$ for all sufficiently small $\lambda>0$.
\end{proof}

\section{Proof of the nonexistence result}

%Let $N=4,5,6$. We consider a solution $u_\lambda$ of (\ref{PBN}) of the form $u_\lambda=PU_{\delta_1,\xi}-PU_{\delta_2,\xi}+w_\lambda$, and such that $\delta_2=o(\delta_1)$, as $\lambda \rightarrow 0$.
We begin considering the case $N=5,6$ since the case $N=4$ requires different estimates.

%\begin{teo}\label{mainteo}
%Let $N=5,6$, let $\Omega \subset \R^N$ be a bounded domain with smooth boundary and let $\xi \in \Omega$. It cannot exist a family of solutions $(u_\lambda)$ of (\ref{PBN}) of the form $$u_\lambda=PU_{\delta_1,\xi}-PU_{\delta_2,\xi}+w_\lambda,$$ defined for $\lambda \rightarrow 0$, such that $\delta_2=o(\delta_1)$ as $\lambda \rightarrow 0$, and $|w_\lambda|=o(\delta_1^{-\frac{N-2}{2}})$, $|\nabla w_\lambda|=o( \delta_1^{-\frac{N}{2}})$,  uniformly in compact subsets of $\Omega$.
%\end{teo}
% \delta_2^{-\frac{1}{2}}
\begin{proof}[\bf{Proof of Theorem \ref{teoprincipintro} for N=5,6}]
Arguing by contradiction let us assume that such a family of solutions exists and, without loss of generality set $\xi=0$. Defining $r:=\sqrt{\delta_1\delta_2}$, we apply the Pohozaev formula (\ref{yylp}) to $u_\lambda$ in the ball $B_{r}=B_r(0)$. Since $u_\lambda$ is a solution of (\ref{PBN}) we set $f(u):=\lambda u + |u|^{p-1}u$ and hence, using the notation of Proposition \ref{yylpohozaev}, we have $F(u)=\frac{\lambda}{2}u^2+\frac{1}{p+1}|u|^{p+1}$.  By elementary computations \footnote{\begin{eqnarray*} NF(u)-\frac{N-2}{2}u f(u) &=& N\left(\frac{\lambda}{2}u^2 + \frac{1}{p+1}|u|^{p+1}\right)-\frac{N-2}{2}(\lambda u^2+|u|^{p+1})\\
 &=&\left(\frac{N}{2}-\frac{N-2}{2}\right)\lambda u^2 + \left(\frac{N}{p+1}-\frac{N-2}{2}\right)|u|^{p+1}\\ 
 &=&\lambda u^2.
 \end{eqnarray*}} (see the footnote)
we get that the left-hand side of (\ref{yylp}) reduces to
 $$\lambda \int_{B_{r}} u_\lambda^2 \ dx.$$ For the right-hand side
 
 $$\displaystyle \int_{\partial B_{r}} \left\{ \sum_{i=1}^N x_i \nu_i \left(F(u_\lambda)-\frac{1}{2}|\nabla u_\lambda|^2\right)+ \frac{\partial u_\lambda}{\partial \nu}\sum_{i=1}^N x_i \frac{\partial u_\lambda}{\partial x_i}+ \frac{N-2}{2}u_\lambda\frac{\partial u_\lambda}{\partial \nu}\right\} \ d\sigma,
 $$
since $\partial B_{r}$ is a sphere, we have $\nu_i(x)=\frac{x_i}{|x|}$ for all $x \in \partial B_{r}$, $i=1,\ldots,N$, and hence $ \sum_{i=1}^N x_i \nu_i=|x|$. Furthermore since $ \frac{\partial u_\lambda}{\partial \nu}=\nabla u_\lambda \cdot \frac{x}{|x|} $ and $\sum_{i=1}^N x_i \frac{\partial u_\lambda}{\partial x_i}=\left( \nabla u_\lambda \cdot \frac{x}{|x|}\right)|x|$ we get that $$ \frac{\partial u_\lambda}{\partial \nu}\sum_{i=1}^N x_i \frac{\partial u_\lambda}{\partial x_i}=\left( \nabla u_\lambda \cdot \frac{x}{|x|}\right) \sum_{i=1}^N x_i \frac{\partial u_\lambda}{\partial x_i}=\left( \nabla u_\lambda \cdot \frac{x}{|x|}\right)^2|x|,$$
$$ u_\lambda\frac{\partial u_\lambda}{\partial \nu}= u_\lambda\left( \nabla u_\lambda \cdot \frac{x}{|x|}\right).$$
Thus (\ref{yylp}) rewrites as
\begin{equation}\label{yylpr}
\begin{array}{lll}
&&\displaystyle  \lambda \int_{B_{r}} u_\lambda^2 \ dx \\[12pt]
&=&\displaystyle \int_{\partial B_{r}} \left\{ |x| \left(F(u_\lambda)-\frac{1}{2}|\nabla u_\lambda|^2\right)+ \left( \nabla u_\lambda \cdot \frac{x}{|x|}\right)^2|x|+ \frac{N-2}{2}u_\lambda \left( \nabla u_\lambda \cdot \frac{x}{|x|}\right) \right\} \ d\sigma.
\end{array}
\end{equation}
We estimate the left-hand side of (\ref{yylpr}). Let us fix a compact subset $K \subset \Omega$; for $\lambda>0$ sufficiently small we get that $B_r \subset K$.
Thanks to Lemma \ref{svilprbubb}  we have
$PU_{\delta_j}=U_{\delta_j}-\varphi_{\delta_j}$, where $\varphi_{\delta_j}=O\left(\delta_j^{\frac{N-2}{2}}\right)$, for $j=1,2$, and this estimate is uniform for $x \in K$, in particular for $x \in B_{r}$. Thus, as $\lambda \rightarrow 0$, we get that
\begin{equation} \label{stimappp}
\begin{array}{lll}
\displaystyle \lambda \int_{B_{r}} u_\lambda^2 \ dx&=&\displaystyle \lambda \int_{B_{r}} \left(PU_{\delta_1}-PU_{\delta_2}+o( \delta_1^{-\frac{N-2}{2}})\right)^2 \ dx\\[12pt]
&=&\displaystyle\lambda \int_{B_{r}} \left(U_{\delta_1}-U_{\delta_2}-\varphi_{\delta_1}+\varphi_{\delta_2}+o( \delta_1^{-\frac{N-2}{2}})\right)^2 \ dx\\[12pt]
&=&\displaystyle\lambda \int_{B_{r}} \left(U_{\delta_1}-U_{\delta_2}+o( \delta_1^{-\frac{N-2}{2}})\right)^2 \ dx\\[12pt]
&=&\displaystyle\lambda \int_{B_{r}}\left(U_{\delta_1}^2+U_{\delta_2}^2 - 2U_{\delta_1}U_{\delta_2}+o(\delta_1^{-\frac{N-2}{2}}U_{\delta_1})+o(\delta_1^{-\frac{N-2}{2}}U_{\delta_2})+o( \delta_1^{-\frac{N-2}{2}})\right) \ dx\\[12pt]
&=&A+B+C+D+E+F.
\end{array}
\end{equation}
%Assume first that $N=5$ or $N=6$. 
We estimate every term of the previous decomposition. %The case $N=4$ requires different estimates (for the term $B$ only) and it will be discussed.

\begin{equation*}
\begin{array}{lllll}
\displaystyle A&=&\displaystyle \lambda \int_{B_{r}}\alpha_N^2 \frac{\delta_1^{N-2}}{(\delta_1^2+|x|^2)^{N-2}} \ dx
&=&\displaystyle \alpha_N^2 \lambda \int_{B_{r}} \frac{\delta_1^{-(N-2)}}{(1+|x/\delta_1|^2)^{N-2}} \ dx\\[12pt]
%&=&\displaystyle \alpha_N^2 \lambda \int_{B_{r}/\delta_1} \frac{\delta_1^{-(N-2)}}{(1+|y|^2)^{N-2}} \delta_1^N\ dy
&=&\displaystyle \alpha_N^2 \lambda \delta_1^2 \int_{B_{r}/\delta_1} \frac{1}{(1+|y|^2)^{N-2}}\ dy&\leq&\displaystyle \alpha_N^2 \lambda \delta_1^2 |B_{r}/\delta_1|\\[12pt]
&=&c_N \lambda \delta_1^2 \left(\frac{\delta_2}{\delta_1}\right)^{\frac{N}{2}},
\end{array}
\end{equation*}
where we have set $c_N:=\alpha_N^2 \frac{\omega_N}{N}$, $\omega_N$ is the measure of the $(N-1)$-dimensional unit sphere $\mathbb{S}^{N-1}$.
\begin{equation*}
\begin{array}{lllll}
\displaystyle B&=&\displaystyle \lambda \int_{B_{r}}\alpha_N^2 \frac{\delta_2^{N-2}}{(\delta_2^2+|x|^2)^{N-2}} \ dx
=\displaystyle \alpha_N^2 \lambda \int_{B_{r}} \frac{\delta_2^{-(N-2)}}{(1+|x/\delta_2|^2)^{N-2}} \ dx\\[12pt]
%&=&\displaystyle \alpha_N^2 \lambda \int_{B_{r}/\delta_2} \frac{\delta_2^{-(N-2)}}{(1+|y|^2)^{N-2}} \delta_2^N\ dy
&=&\displaystyle \alpha_N^2 \lambda \delta_2^2 \int_{B_{r}/\delta_2} \frac{1}{(1+|y|^2)^{N-2}}\ dy\\[12pt]
&=&\displaystyle \alpha_N^2 \lambda \delta_2^2  \int_{\R^N} \frac{1}{(1+|y|^2)^{N-2}}\ dy + O\left( \lambda \delta_2^2 \int_{\left(\frac{\delta_1}{\delta_2}\right)^{\frac{1}{2}}}^{+\infty} \frac{r^{N-1}}{(1+r^2)^{N-2}}\ dr\right)\\[12pt]
&=&a_1 \lambda \delta_2^2 + O\left(  \lambda \delta_2^2 \left(\frac{\delta_2}{\delta_1}\right)^{\frac{N-4}{2}}\right),
\end{array}
\end{equation*}
where we have set $a_1:= \alpha_N^2 \int_{\R^N} \frac{1}{(1+|y|^2)^{N-2}}\ dy$. We point out that since $N=5$ or $N=6$  the function $\frac{1}{(1+|y|^2)^{N-2}} \in L^1(\R^N)$ while this is not true when $N=4$. 

\begin{eqnarray*}
\begin{array}{lll}
|C| &=&\displaystyle\lambda \ \alpha_N^2 \int_{B_{r}} \frac{\delta_1^{\frac{N-2}{2}}}{(\delta_1^2+|x|^2)^{\frac{N-2}{2}}}  \frac{\delta_2^{\frac{N-2}{2}}}{(\delta_2^2+|x|^2)^{\frac{N-2}{2}}} \ dx\\[12pt]
%&=&\displaystyle\lambda \ \alpha_N^2 \int_{B_{r}} \frac{\delta_1^{-\frac{N-2}{2}}}{(1+|x/\delta_1|^2)^{\frac{N-2}{2}}}  \frac{\delta_2^{\frac{N-2}{2}}}{(\delta_2^2+|x|^2)^{\frac{N-2}{2}}} \ dx\\[12pt]
&=&\displaystyle\lambda \ \alpha_N^2 \int_{B_{r}/\delta_1} \frac{\delta_1^{\frac{N+2}{2}}}{(1+|y|^2)^{\frac{N-2}{2}}}  \frac{\delta_2^{\frac{N-2}{2}}}{(\delta_2^2+\delta_1^2|y|^2)^{\frac{N-2}{2}}} \ dy\\[12pt]
&=&\displaystyle\lambda \ \alpha_N^2 \int_{B_{r}/\delta_1} \frac{\delta_1^{-\frac{N-6}{2}}}{(1+|y|^2)^{\frac{N-2}{2}}}  \frac{\delta_2^{\frac{N-2}{2}}}{\left(\left(\frac{\delta_2}{\delta_1}\right)^2+|y|^2\right)^{\frac{N-2}{2}}} \ dy\\[12pt]
&\leq&\displaystyle\lambda \ \alpha_N^2 \left(\frac{\delta_2}{\delta_1}\right)^{\frac{N-2}{2}} \delta_1^2 \int_{B_{r}/\delta_1} \frac{1}{(1+|y|^2)^{\frac{N-2}{2}}|y|^{N-2}} \ dy \\[12pt]
&=& \displaystyle O\left(\lambda \left(\frac{\delta_2}{\delta_1}\right)^{\frac{N-2}{2}} \delta_1^2 \int_{0}^{\left(\frac{\delta_2}{\delta_1}\right)^{1/2}} \frac{r^{N-1}}{(1+r^2)^{\frac{N-2}{2}}r^{N-2}} \ dr\right)\\[14pt]
%&=& \displaystyle O\left(\lambda \left(\frac{\delta_2}{\delta_1}\right)^{\frac{N-2}{2}} \delta_1^2 \int_{0}^{\left(\frac{\delta_2}{\delta_1}\right)^{\frac{1}{2}}} r \ dr\right)\\[14pt]
&=& \displaystyle O\left(\lambda \left(\frac{\delta_2}{\delta_1}\right)^{\frac{N}{2}} \delta_1^2 \right).
\end{array}
\end{eqnarray*}

\begin{eqnarray*}
\begin{array}{lll}
|D| &=&\displaystyle o\left(\lambda \delta_1^{-\frac{N-2}{2}}\ \int_{B_{r}} \frac{\delta_1^{\frac{N-2}{2}}}{(\delta_1^2+|x|^2)^{\frac{N-2}{2}}}  \ dx\right)\\[12pt]
&\leq&\displaystyle o\left(\lambda  \int_{B_{r}} \delta_1^{-(N-2)}  \ dx\right)\\[12pt]
%&=&\displaystyle O\left(\lambda \delta_1^{-\frac{N-2}{2}}  |B_{r}|\right)\\[12pt]
&=&\displaystyle o\left(\lambda \delta_1^2  \left(\frac{\delta_2}{\delta_1}\right)^{\frac{N}{2}} \right).
\end{array}
\end{eqnarray*}

\begin{eqnarray*}
\begin{array}{lll}
|E| &=&\displaystyle o\left(\lambda \delta_1^{-\frac{N-2}{2}}  \int_{B_{r}} \frac{\delta_2^{\frac{N-2}{2}}}{(\delta_2^2+|x|^2)^{\frac{N-2}{2}}}  \ dx\right)\\[12pt]
&\leq&\displaystyle o\left(\lambda \delta_1^{-\frac{N-2}{2}} \int_{B_{r}} \frac{\delta_2^{\frac{N-2}{2}}}{|x|^{N-2}} \ dx\right)\\[12pt]
%&=&\displaystyle O\left(\lambda \delta_2^{\frac{N-2}{2}} \int_{0}^{\sqrt{\delta_1\delta_2}} r \ dr  \right)\\[12pt]
&=&\displaystyle o\left(\lambda \left(\frac{\delta_2}{\delta_1}\right)^{\frac{N}{2}} \right).
\end{array}
\end{eqnarray*}

\begin{eqnarray*}
\begin{array}{lll}
|F| &=&\displaystyle o\left(\lambda \delta_1^{-\frac{N-2}{2}} |B_{r}|\right)\\[12pt]
&=&\displaystyle o\left(\lambda\ \delta_1 \ \delta_2^{\frac{N}{2}} \right).
\end{array}
\end{eqnarray*}

Now we estimate the right-hand side of (\ref{yylpr}). Remembering that $F(u_\lambda)=\frac{\lambda}{2}u_\lambda^2+\frac{1}{p+1}|u_\lambda|^{p+1}$ we get that the first term is equal to
$$\int_{\partial B_{r}}  |x| \left(\frac{\lambda}{2}u_\lambda^2+\frac{1}{p+1}|u_\lambda|^{p+1}-\frac{1}{2}|\nabla u_\lambda|^2\right) \ d\sigma.$$
We observe that by definition of $r$ it is immediate to see that $$U_{\delta_1}(x)=U_{\delta_2}(x),$$  for all $x \in \partial B_r$, and hence we have

\begin{eqnarray*}
\begin{array}{lll}
\displaystyle  \int_{\partial B_{r}}  \frac{\lambda}{2}u_\lambda^2\ |x| \ d\sigma
&=&\displaystyle\frac{\lambda}{2} \int_{\partial B_{r}} \left(U_{\delta_1}-U_{\delta_2}+o\left(\delta_1^{-\frac{N-2}{2}}\right)\right)^2 \ |x|\ d\sigma\\[12pt]
&=&\displaystyle\frac{\lambda}{2} \int_{\partial B_{r}} \left[o\left(\delta_1^{-\frac{N-2}{2}}\right)\right]^2 \ |x|\ d\sigma\\[12pt]
&=&\displaystyle o\left( \lambda \delta_1^{-(N-2)} \int_{\partial B_{r}} |x|\ d\sigma\right)\\[12pt]
&=&\displaystyle o\left(\lambda \left(\frac{\delta_2}{\delta_1}\right)^{\frac{N}{2}} \delta_1^2\right).
\end{array}
\end{eqnarray*}
%
%We estimate now the term
%$$ \frac{1}{p+1}\int_{\partial B_{r}}|u_\lambda|^{p+1}|x| \ d\sigma.$$
As in the previous case we have

\begin{eqnarray*}
\begin{array}{lll}
\displaystyle  \frac{1}{p+1}\int_{\partial B_{r}}|u_\lambda|^{p+1}|x| \ d\sigma
&=&\displaystyle\frac{1}{p+1} \int_{\partial B_{r}} |U_{\delta_1}-U_{\delta_2}+o(\delta_1^{-\frac{N-2}{2}})|^{p+1} \ |x|\ d\sigma\\[12pt]
&=&\displaystyle\frac{1}{p+1} \int_{\partial B_{r}} |o(\delta_1^{-\frac{N-2}{2}})|^{p+1} \ |x|\ d\sigma\\[12pt]
&=&\displaystyle o\left(\delta_1^{-N}\int_{\partial B_{r}} |x|\ d\sigma\right)\\[12pt]
&=&\displaystyle o\left(\left(\frac{\delta_2}{\delta_1}\right)^{\frac{N}{2}}\right).
\end{array}
\end{eqnarray*}

To complete the estimate of the first term it remains to analyze $$-\frac{1}{2}\int_{\partial B_{r}} |\nabla u_\lambda|^2 |x|  \ d\sigma.$$
%
%\begin{equation*}
%\int_{\partial B_{r}}  \left( \frac{\lambda}{2}u_\lambda^2+\frac{1}{p+1}|u_\lambda|^{p+1}\right) |x| \ d\sigma \geq0,
%\end{equation*}
 %
%Thus we analyze $$-\frac{1}{2}\int_{\partial B_{r}} |\nabla u_\lambda|^2 |x|  \ d\sigma.$$
%
As before, writing $PU_{\delta_j}=U_{\delta_j}- \varphi_{\delta_j}$ for $j=1,2$ we have $$|\nabla u_\lambda|^2=|\nabla U_{\delta_1} - \nabla U_{\delta_2} - \nabla \varphi_{\delta_1}+ \nabla \varphi_{\delta_2}+ \nabla w_\lambda|^2=| \nabla U_{\delta_1} - \nabla U_{\delta_2} + \nabla \Phi_\lambda|^2, $$
where we have set $ \Phi_\lambda:=- \varphi_{\delta_1}+ \varphi_{\delta_2}+ w_\lambda$. 
Hence, we get that

\begin{equation}\label{eqdec3}
\begin{array}{lll}
&&\displaystyle -\frac{1}{2}\int_{\partial B_{r}} |\nabla u_\lambda|^2 |x|  \ d\sigma\\[16pt]
%&=&\displaystyle  -\frac{1}{2}\int_{\partial B_{r}}\left(| \nabla U_{\delta_1}|^2 + | \nabla U_{\delta_2}|^2-2\nabla U_{\delta_1} \cdot \nabla U_{\delta_2}+ 2\nabla U_{\delta_1} \cdot \nabla \Phi -  2\nabla U_{\delta_2} \cdot \nabla \Phi + | \nabla \Phi|^2\right) |x| \ d\sigma\\[16pt]
&=&\displaystyle  -\frac{1}{2}\int_{\partial B_{r}}| \nabla U_{\delta_1}|^2\ |x| \ d\sigma  -\frac{1}{2}\int_{\partial B_{r}} | \nabla U_{\delta_2}|^2\ |x| \ d\sigma + \int_{\partial B_{r}} \nabla U_{\delta_1} \cdot \nabla U_{\delta_2}\ |x| \ d\sigma \\[16pt] 
&&\displaystyle -  \int_{\partial B_{r}} \nabla U_{\delta_1} \cdot \nabla \Phi_\lambda\ |x| \ d\sigma  +  \int_{\partial B_{r}} \nabla U_{\delta_2} \cdot \nabla \Phi_\lambda\ |x| \ d\sigma-\frac{1}{2}  \int_{\partial B_{r}} | \nabla \Phi_\lambda|^2\ |x| \ d\sigma\\[16pt]
&=&A_1+B_1+C_1+D_1+E_1+F_1.
\end{array}
\end{equation}

By elementary computations, for all $i=1,\ldots,N$, $j=1,2$ we have:
\begin{eqnarray*}
\frac{\partial U_{\delta_j}}{\partial x_i}(x)&=&-\alpha_N (N-2)\delta_j^{\frac{N-2}{2}} \frac{x_i}{(\delta_j^2+|x|^2)^{\frac{N}{2}}},
\end{eqnarray*}
\begin{equation}\label{normal2udelta}
|\nabla U_{\delta_j}|^2=\alpha_N^2 (N-2)^2\delta_j^{N-2} \frac{|x|^2}{(\delta_j^2+|x|^2)^{N}}.
\end{equation}
Thus, we get that
\begin{equation*}
\begin{array}{lll}
A_1%&=&\displaystyle  -\alpha_N^2\frac{(N-2)^2}{2}\frac{\delta_1^{N-2}}{(\delta_1^2+\delta_1 \delta_2)^N} \int_{\partial B_{r}} |x|^3 \ d\sigma\\[16pt]
&=&\displaystyle  -\alpha_N^2\frac{(N-2)^2}{2}\frac{\delta_1^{-(N+2)}}{\left[1+\left(\frac{\delta_2}{\delta_1}\right)\right]^N} \int_{\partial B_{r}} |x|^3 \ d\sigma\\[16pt]
&=&\displaystyle  -\alpha_N^2\frac{(N-2)^2}{2}\omega_N\frac{\delta_1^{-(N+2)}}{\left[1+\left(\frac{\delta_2}{\delta_1}\right)\right]^N} \delta_1^{\frac{N+2}{2}}\delta_2^{\frac{N+2}{2}}\\[24pt]
&=&\displaystyle  -\alpha_N^2\frac{(N-2)^2}{2}\omega_N\left(\frac{\delta_2}{\delta_1}\right)^{\frac{N+2}{2}} +O\left(\left(\frac{\delta_2}{\delta_1}\right)^{\frac{N+4}{2}}\right). \\[16pt]
\end{array}
\end{equation*}

\begin{equation*}
\begin{array}{lll}
B_1%&=&\displaystyle-  \alpha_N^2\frac{(N-2)^2}{2}\frac{\delta_2^{N-2}}{(\delta_2^2+\delta_1\delta_2)^N} \int_{\partial B_{r}} |x|^3 \ d\sigma\\[16pt]
&=&\displaystyle -\alpha_N^2\frac{(N-2)^2}{2}\frac{\delta_2^{N-2}\delta_1^{-N}\delta_2^{-N}}{\left[1+\left(\frac{\delta_2}{\delta_1}\right)\right]^N} \int_{\partial B_{r}} |x|^3 \ d\sigma\\[26pt]
%&=&\displaystyle - \alpha_N^2\frac{(N-2)^2}{2}\omega_N\frac{\delta_2^{N-2}\delta_1^{-N}\delta_2^{-N}}{\left[1+\left(\frac{\delta_2}{\delta_1}\right)\right]^N} \delta_1^{\frac{N+2}{2}}\delta_2^{\frac{N+2}{2}}\\[24pt]
&=&\displaystyle  -\alpha_N^2\frac{(N-2)^2}{2}\omega_N\left(\frac{\delta_2}{\delta_1}\right)^{\frac{N-2}{2}} +O\left(\left(\frac{\delta_2}{\delta_1}\right)^{\frac{N}{2}}\right). 
\end{array}
\end{equation*}

\begin{equation*}
\begin{array}{lll}
C_1%&=&\displaystyle \alpha_N^2 (N-2)^2\frac{\delta_1^{\frac{N-2}{2}}\delta_2^{\frac{N-2}{2}}}{(\delta_1^2+\delta_1\delta_2)^{\frac{N}{2}}(\delta_2^2+\delta_1\delta_2)^{\frac{N}{2}}} \int_{\partial B_{r}} |x|^3 \ d\sigma\\[26pt]
&=&\displaystyle \alpha_N^2 (N-2)^2\frac{\delta_1^{\frac{N-2}{2}}\delta_2^{\frac{N-2}{2}}\delta_1^{-N}\delta_1^{-\frac{N}{2}}\delta_2^{-\frac{N}{2}}}{\left[1+\left(\frac{\delta_2}{\delta_1}\right)\right]^{\frac{N}{2}}\left[1+\left(\frac{\delta_2}{\delta_1}\right)\right]^{\frac{N}{2}}} \int_{\partial B_{r}} |x|^3 \ d\sigma\\[32pt]
%&=&\displaystyle \alpha_N^2 (N-2)^2\omega_N\frac{\delta_1^{\frac{N-2}{2}}\delta_2^{\frac{N-2}{2}}\delta_1^{-N}\delta_1^{-\frac{N}{2}}\delta_2^{-\frac{N}{2}}} {\left[1+\left(\frac{\delta_2}{\delta_1}\right)\right]^{N}}\delta_1^{\frac{N+2}{2}}\delta_2^{\frac{N+2}{2}}\\[32pt]
&=&\displaystyle \alpha_N^2 (N-2)^2\omega_N\frac{\left(\frac{\delta_2}{\delta_1}\right)^{\frac{N}{2}}} {\left[1+\left(\frac{\delta_2}{\delta_1}\right)\right]^{N}}\\[32pt]
&=&\displaystyle \alpha_N^2 (N-2)^2\omega_N{\left(\frac{\delta_2}{\delta_1}\right)^{\frac{N}{2}}} +O\left(\left(\frac{\delta_2}{\delta_1}\right)^{\frac{N+2}{2}}\right).
\end{array}
\end{equation*}
Taking into account the assumptions on the remainder term $w_\lambda$ and thanks to Lemma  \ref{svilprbubb} we have $|\nabla \Phi_\lambda| = o(\delta_1^{-\frac{N}{2}})$, %for some small $\sigma>0$, 
uniformly on $\partial B_r$. Thus we have the following:
\begin{equation*}
\begin{array}{lll}
|D_1|&\leq&\displaystyle \int_{\partial B_{r}} |\nabla U_{\delta_1}| |\nabla \Phi_\lambda| |x| \ d\sigma\\[16pt]
&=&\displaystyle o\left(\frac{\delta_1^{\frac{N-2}{2}}}{(\delta_1^2+\delta_1\delta_2)^{\frac{N}{2}}}\delta_1^{-\frac{N}{2}} \int_{\partial B_{r}}|x|^2 \ d\sigma \right)\\[16pt]
&=&\displaystyle o\left(\frac{\delta_1^{\frac{N-2}{2}}\delta_1^{-N}}{\left[1+\left(\frac{\delta_2}{\delta_1}\right)\right]^{\frac{N}{2}}} \delta_1^{-\frac{N}{2}}\int_{\partial B_{r}}|x|^2 \ d\sigma \right)\\[26pt]
%&=&\displaystyle O\left({\delta_1^{-\frac{N+2}{2}}} \delta_1^{-\frac{N}{2}}\ \delta_1^{\frac{N+1}{2}}\delta_2^{\frac{N+1}{2}} \right)\\[16pt]
&=&\displaystyle o\left(\left(\frac{\delta_2}{\delta_1} \right)^{\frac{N+1}{2}}\right).\\[16pt]
\end{array}
\end{equation*}

\begin{equation*}
\begin{array}{lll}
|E_1|&\leq&\displaystyle \int_{\partial B_{r}} |\nabla U_{\delta_2}| |\nabla \Phi_\lambda| |x| \ d\sigma\\[16pt]
%&=&\displaystyle O\left(\frac{\delta_2^{\frac{N-2}{2}}}{(\delta_2^2+\delta_1\delta_2)^{\frac{N}{2}}} \delta_1^{-\frac{N}{2}} \int_{\partial B_{r}}|x|^2 \ d\sigma \right)\\[16pt]
&=&\displaystyle o\left(\frac{\delta_2^{\frac{N-2}{2}}\delta_1^{-\frac{N}{2}}\delta_2^{-\frac{N}{2}}}{\left[1+\left(\frac{\delta_2}{\delta_1}\right)\right]^{\frac{N}{2}}} \delta_1^{-\frac{N}{2}} \int_{\partial B_{r}}|x|^2 \ d\sigma \right)\\[26pt]
%&=&\displaystyle O\left(\delta_2^{\frac{N-2}{2}}\delta_1^{-\frac{N}{2}}\delta_2^{-\frac{N}{2}} \  \delta_1^{-\frac{N}{2}}\delta_1^{\frac{N+1}{2}}\delta_2^{\frac{N+1}{2}} \right)\\[16pt]
&=&\displaystyle o\left(\left(\frac{\delta_2}{\delta_1} \right)^{\frac{N-1}{2}} \right).
\end{array}
\end{equation*}
And finally the last term of (\ref{eqdec3}) is trivial:
\begin{equation*}
\begin{array}{lllll}
\displaystyle |F_1|&=&\displaystyle o\left( \left(\frac{\delta_2}{\delta_1}\right)^{\frac{N}{2}} \right).
\end{array}
\end{equation*}

Now we analyze the term 
\begin{equation}\label{terzopezzo}
\int_{\partial B_{r}} \left( \nabla u_\lambda \cdot \frac{x}{|x|}\right)^2|x|\ d\sigma. 
\end{equation}
As before we write $u_\lambda=U_{\delta_1}-U_{\delta_2}+ \Phi_\lambda$ and we have
\begin{equation}
\begin{array}{lll}
\displaystyle \left( \nabla u_\lambda \cdot \frac{x}{|x|}\right)^2|x|%&=&\displaystyle \left( \nabla U_{\delta_1} \cdot \frac{x}{|x|} -  \nabla U_{\delta_2} \cdot \frac{x}{|x|} + \nabla \Phi \cdot \frac{x}{|x|} \right)^2|x|\\[12pt]
&=&\displaystyle \left( \nabla U_{\delta_1} \cdot \frac{x}{|x|}\right)^2|x| +  \left( \nabla U_{\delta_2} \cdot \frac{x}{|x|}\right)^2|x|-2  \left( \nabla U_{\delta_1} \cdot \frac{x}{|x|}\right) \left( \nabla U_{\delta_2} \cdot \frac{x}{|x|}\right)|x|\\[12pt]
&&\displaystyle +2  \left( \nabla U_{\delta_1} \cdot \frac{x}{|x|}\right) \left( \nabla \Phi_\lambda \cdot \frac{x}{|x|}\right)|x| -2  \left( \nabla U_{\delta_2} \cdot \frac{x}{|x|}\right) \left( \nabla \Phi_\lambda \cdot \frac{x}{|x|}\right)|x|\\[12pt]
&&\displaystyle+ \left( \nabla \Phi_\lambda \cdot \frac{x}{|x|}\right)^2|x| 
 \end{array}
\end{equation}
By elementary computations we see that for $j=1,2$
\begin{equation*}
\begin{array}{lll}
\displaystyle \left( \nabla U_{\delta_j} \cdot \frac{x}{|x|}\right)^2|x| &=&\displaystyle |\nabla U_{\delta_j}|^2\ |x|,\\[12pt]
\displaystyle  -2 \left( \nabla U_{\delta_1} \cdot \frac{x}{|x|}\right) \left( \nabla U_{\delta_2} \cdot \frac{x}{|x|}\right)\ |x| &=&\displaystyle -2 (\nabla U_{\delta_1} \cdot \nabla U_{\delta_2}) \ |x|,
 \end{array}
\end{equation*}
and for the remaining terms we have
\begin{equation*}
\begin{array}{lll}
\displaystyle \left|\ \pm 2  \left( \nabla U_{\delta_j} \cdot \frac{x}{|x|}\right) \left( \nabla \Phi_\lambda \cdot \frac{x}{|x|}\right)|x|\right| &\leq&\displaystyle2 |\nabla U_{\delta_j}| |\nabla \Phi_\lambda| |x|,\\[12pt]
\displaystyle \left|\left( \nabla \Phi_\lambda \cdot \frac{x}{|x|}\right)^2|x|  \right|&\leq& \left|\nabla \Phi_\lambda \right|^2 |x|.
 \end{array}
\end{equation*}

Thus, in order to estimate (\ref{terzopezzo}) it suffices to apply the estimates of the previous case, and hence we get that
$$\int_{\partial B_{r}} \left( \nabla u_\lambda \cdot \frac{x}{|x|}\right)^2|x|\ d\sigma = \alpha_N^2{(N-2)^2}\omega_N\left(\frac{\delta_2}{\delta_1}\right)^{\frac{N-2}{2}} +o\left(\left(\frac{\delta_2}{\delta_1} \right)^{\frac{N-2}{2}} \right). $$

To complete our analysis of (\ref{yylpr}) it remains only to study the term $$\frac{N-2}{2}\int_{\partial B_{r}}u_\lambda \left( \nabla u_\lambda \cdot \frac{x}{|x|}\right) \ d\sigma.$$ 

%To this end we write the following:
\begin{equation}\label{remainderderim}
\begin{array}{lll}
&&\displaystyle \frac{N-2}{2}\int_{\partial B_{r}}u_\lambda \left( \nabla u_\lambda \cdot \frac{x}{|x|}\right) \ d\sigma\\[12pt] 
&=&\displaystyle \frac{N-2}{2}\int_{\partial B_{r}} (U_{\delta_1}-U_{\delta_2}+ \Phi_\lambda)\left[(\nabla U_{\delta_1} - \nabla U_{\delta_2} + \nabla \Phi_\lambda)\cdot \frac{x}{|x|}\right]  \ d\sigma \\[12pt]
%&=&\displaystyle \frac{N-2}{2}\int_{\partial B_{r}}  \Phi_\lambda\left[(\nabla U_{\delta_1} - \nabla U_{\delta_2} + \nabla \Phi_\lambda)\cdot \frac{x}{|x|}\right] \ d\sigma \\[12pt]
&=&\displaystyle\ \frac{N-2}{2}\int_{\partial B_{r}} \Phi_\lambda\left(\nabla U_{\delta_1}\cdot \frac{x}{|x|}\right)  \ d\sigma - \frac{N-2}{2}\int_{\partial B_{r}} \Phi_\lambda\left(\nabla U_{\delta_2}\cdot \frac{x}{|x|}\right)   \ d\sigma\\[12pt]
&&\displaystyle+ \frac{N-2}{2}\int_{\partial B_{r}} \Phi_\lambda\left(\nabla \Phi_\lambda\cdot \frac{x}{|x|}\right)  \ d\sigma\\[12pt]
&=& A_2+B_2+C_2.
\end{array}
\end{equation}

\begin{equation*}
\begin{array}{lll}
|A_2|&\leq&\displaystyle  \alpha_N^2\frac{(N-2)^2}{2}\frac{\delta_1^{\frac{N-2}{2}}\delta_1^{-N}}{\left[1+\left(\frac{\delta_2}{\delta_1}\right)\right]^{\frac{N}{2}}} \int_{\partial B_{r}} |\Phi_\lambda|\ |x| \ d\sigma\\[26pt]
&=&\displaystyle  o\left(\frac{\delta_1^{\frac{N-2}{2}}\delta_1^{-{N}}}{\left[1+\left(\frac{\delta_2}{\delta_1}\right)\right]^{\frac{N}{2}}} \int_{\partial B_{r}}\delta_1^{-\frac{N-2}{2}}\ |x| \ d\sigma \right)\\[26pt]
&=&\displaystyle  o\left(\frac{\delta_1^{-{N}}}{\left[1+\left(\frac{\delta_2}{\delta_1}\right)\right]^{\frac{N}{2}}} \delta_1^{\frac{N}{2}}\delta_2^{\frac{N}{2}} \right)\\[26pt]
%&=&\displaystyle  O\left(\delta_1^{-1}\delta_2^{\frac{N}{2}} \right)\\[16pt]
&=&\displaystyle  o\left(\left(\frac{\delta_2}{\delta_1}\right)^{\frac{N}{2}}\right).
\end{array}
\end{equation*}

\begin{equation*}
\begin{array}{lll}
|B_2|
%&=&\displaystyle  \alpha_N^2\frac{(N-2)^2}{2}\frac{\delta_2^{\frac{N-2}{2}}}{(\delta_1^2+\delta_1\delta_2)^{\frac{N}{2}}} \int_{\partial B_{r}} \Phi_\lambda\ |x| \ d\sigma\\[16pt]
&\leq&\displaystyle  \alpha_N^2\frac{(N-2)^2}{2}\frac{\delta_2^{\frac{N-2}{2}}\delta_1^{-\frac{N}{2}}\delta_2^{-\frac{N}{2}}}{\left[1+\left(\frac{\delta_2}{\delta_1}\right)\right]^{\frac{N}{2}}} \int_{\partial B_{r}} |\Phi_\lambda|\ |x| \ d\sigma\\[26pt]
&=&\displaystyle  o\left(\frac{\delta_2^{\frac{N-2}{2}}\delta_1^{-\frac{N}{2}}\delta_2^{-\frac{N}{2}}}{\left[1+\left(\frac{\delta_2}{\delta_1}\right)\right]^{\frac{N}{2}}} \int_{\partial B_{r}}\delta_1^{-\frac{N-2}{2}}\ |x| \ d\sigma \right)\\[26pt]
%&=&\displaystyle  O\left(\frac{\delta_2^{\frac{N-2}{2}}\delta_1^{-\frac{N}{2}}\delta_2^{-\frac{N}{2}}}{\left[1+\left(\frac{\delta_2}{\delta_1}\right)\right]^{\frac{N}{2}}} \delta_1^{\frac{N}{2}}\delta_2^{\frac{N}{2}} \right)\\[26pt]
&=&\displaystyle  o\left(\left(\frac{\delta_2}{\delta_1}\right)^{\frac{N-2}{2}} \right).
\end{array}
\end{equation*}

\begin{equation*}
\begin{array}{lll}
|C_2|&\leq&\displaystyle  \frac{(N-2)}{2}\int_{\partial B_{r}}|\Phi_\lambda| |\nabla\Phi_\lambda| \ d\sigma\\[16pt]
%&=&\displaystyle O\left(	\delta_1^{-\frac{N}{2}}\int_{\partial B_{r}}1 \ d\sigma\right)\\[16pt]
&=&\displaystyle o\left(\delta_1^{-\frac{N-2}{2}} \delta_1^{-\frac{N}{2}} \delta_1^{\frac{N-1}{2}}\delta_2^{\frac{N-1}{2}}\right)\\[12pt]
&=&\displaystyle  o\left(\left(\frac{\delta_2}{\delta_1}\right)^{\frac{N-1}{2}}\right).
\end{array}
\end{equation*}

Summing up all the estimates, from (\ref{yylp}), for all sufficiently small $\lambda>0$, we deduce the following equation %(which holds for $N=5,6$, and for all sufficiently small $\lambda>0$)
\begin{equation}\label{eqfinalepohozaev}
a_1 \lambda \delta_2^2 + o\left(  \lambda \delta_2^2% \left(\frac{\delta_2}{\delta_1}\right)^{\frac{N-4}{2}}
\right) = \alpha_N^2\frac{(N-2)^2}{2}\omega_N\left(\frac{\delta_2}{\delta_1}\right)^{\frac{N-2}{2}} +o\left(\left(\frac{\delta_2}{\delta_1}\right)^{\frac{N-2}{2}}\right). 
\end{equation}
From (\ref{eqfinalepohozaev}) we deduce that 
\begin{equation}\label{eqfinalepohoza}
a_1 \lambda \delta_1^{\frac{N-2}{2}} (1+o(1))= \alpha_N^2\frac{(N-2)^2}{2}\omega_N \delta_2^{\frac{N-6}{2}}  (1+o(1)), 
\end{equation}
 for all sufficiently small $\lambda>0$. Since $N=5,6$ it is clear that \eqref{eqfinalepohoza} is contradictory, in fact, passing to the limit as $\lambda \rightarrow 0$, the left-hand side goes to zero while the right-hand side goes to a constant, when $N=6$ and diverges to $+\infty$ when $N=5$.
%$ \lambda \delta_2^2$ must be of the same order as $\left(\frac{\delta_2}{\delta_1}\right)^{\frac{N-2}{2}}$. Hence, we get that $$
%Hence, thanks to our assumptions on $\delta_i$ for $i=1,2$, we get that $\alpha_1$, $\alpha_2$ must satisfy the following algebraic equation
%\begin{equation}\label{eqalgebrica}
%1+2\alpha_2 = \frac{N-2}{2}(\alpha_2 - \alpha_1), 
%\end{equation}
%
%which is equivalent to $1+ \frac{N-2}{2}\alpha_1 =\frac{N-6}{2}\alpha_2$, and clearly it cannot be satisfied by any couple $\alpha_1, \alpha_2$ of positive real numbers, since we are assuming $N=5,6$.
The proof is complete.
\end{proof}

Now we turn to the case $N=4$
%\begin{teo}\label{mainteobis}
%Let $N=4$, let $\Omega \subset \R^N$ be a bounded domain with smooth boundary and let $\xi \in \Omega$. It cannot exist a family of solutions $(u_\lambda)$ of (\ref{PBN}) of the form $$u_\lambda=PU_{\delta_1,\xi}-PU_{\delta_2,\xi}+w_\lambda,$$ defined for $\lambda \rightarrow 0$, with $\delta_2=o(\delta_1)$ as $\lambda \rightarrow0$, and where $w_\lambda$ is such that %$\|w_\lambda\|_\Omega \rightarrow 0$, as $\lambda \rightarrow 0$, 
%$|w_\lambda|=o(\delta_1^{-\frac{N-2}{2}})$,  $|\nabla w_\lambda|=o( \delta_1^{-\frac{N}{2}})$, uniformly in compact subsets of $\Omega$. %for all $x$ such that $|x|=\sqrt{\delta_1 \delta_2}$.
%\end{teo}

\begin{proof}[\bf{Proof of Theorem \ref{teoprincipintro} for N=4}]
%We treat now the case $N=4$.
Again, without loss of generality we assume that $\xi=0$. We repeat the scheme of the proof for the previous case, but some modification is needed. In fact, since $N=4$, we have to change the estimate of the term $B$ in (\ref{stimappp}):
 %First, since we don't give conditions on $w_\lambda$ in compact subsets of $\Omega$, but only on the boundary of $B_r$, %and also we assume that $\|w_\lambda\|_\Omega \rightarrow 0$, as $\lambda \rightarrow 0$, 
%we have to repeat all the estimates involving this remainder term in the whole $B_r$. Finally, as observed in the proof of Theorem \ref{mainteo}, since $N=4$, we have to change the estimate of the term $B$ in (\ref{stimappp}).
 %As before we set $\Phi_\lambda:=-\varphi_{\delta_1}+\varphi_{\delta_2}+w_\lambda$. %Instead of \eqref{stimappp}
%we have 
 %\begin{equation} \label{stimappp*}
%\begin{array}{lll}
%\displaystyle \lambda \int_{B_{r}} u_\lambda^2 \ dx&=&\displaystyle \lambda \int_{B_{r}} (PU_{\delta_1}-PU_{\delta_2}+w_\lambda)^2 \ dx\\[12pt]
%&=&\displaystyle\lambda \int_{B_{r}} (U_{\delta_1}-U_{\delta_2}-\varphi_{\delta_1}+\varphi_{\delta_2}+w_\lambda)^2 \ dx\\[12pt]
%&=&\displaystyle\lambda \int_{B_{r}} (U_{\delta_1}-U_{\delta_2}+\Phi_\lambda)^2 \ dx\\[12pt]
%&=&\displaystyle\lambda \int_{B_{r}}\left(U_{\delta_1}^2+U_{\delta_2}^2 - 2U_{\delta_1}U_{\delta_2}+2U_{\delta_1}\Phi_\lambda -2U_{\delta_2}\Phi_\lambda+ \Phi_\lambda^2\right) \ dx\\[12pt]
%&=&A+B_*+C+D_*+E_*+F_*.
%\end{array}
%\end{equation} 
%$A$ and $C$ are the same as in the proof of Theorem \ref{mainteo}.  
\begin{equation*}
\begin{array}{lllll}
\displaystyle B_*&=&\displaystyle \lambda \int_{B_{r}}\alpha_4^2 \frac{\delta_2^{2}}{(\delta_2^2+|x|^2)^{2}} \ dx%=\displaystyle \alpha_4^2 \lambda \int_{B_{r}} \frac{\delta_2^{-2}}{(1+|x/\delta_2|^2)^{2}} \ dx\\[12pt]
=\displaystyle \alpha_4^2 \lambda \int_{B_{r}/\delta_2} \frac{\delta_2^{-2}}{(1+|y|^2)^{2}} \delta_2^4\ dy\\[12pt]
&=&\displaystyle \alpha_4^2 \lambda \delta_2^2 \int_{B_{r}/\delta_2} \frac{1}{(1+|y|^2)^{2}}\ dy
=\displaystyle \alpha_4^2 \omega_4 \lambda \delta_2^2 \int_{0}^{\left(\frac{\delta_1}{\delta_2}\right)} \frac{r^3}{(1+r^2)^{2}}\ dr\\[12pt] 
%&=& \displaystyle \alpha_4^2 \omega_4 \lambda \delta_2^2 \int_{0}^{1} \frac{r^3}{(1+r^2)^{2}}\ dr +  \alpha_4^2 \omega_4 \lambda \delta_2^2 \int_{1}^{\left(\frac{\delta_1}{\delta_2}\right)} \frac{r^3}{(1+r^2)^{2}}\ dr = B_{*} + B_{**}.
\end{array}
\end{equation*}

It's elementary to see that 
$$\int_{0}^{\left(\frac{\delta_1}{\delta_2}\right)} \frac{r^3}{(1+r^2)^{2}}\ dr=O\left( \log\left(\frac{\delta_1}{\delta_2}\right)\right),$$
and hence we have that
\begin{equation}\label{stimaBstar}
B_*=O\left(\lambda \delta_2^2 \log\left(\frac{\delta_1}{\delta_2}\right)\right). 
\end{equation}
Thus, summing up \eqref{stimaBstar} with the other estimates made in the previous case (in which we take $N=4$), from (\ref{yylp}), we deduce the following asymptotic relation

\begin{equation}\label{eqfinalepohozaev2}
O\left(\lambda \delta_2^2 \log\left(\frac{\delta_1}{\delta_2}\right)\right) + o\left(\lambda \delta_2^2 \log\left(\frac{\delta_1}{\delta_2}\right)\right) = 2\alpha_4^2\omega_4 \left(\frac{\delta_2}{\delta_1}\right)+ o\left(\frac{\delta_2}{\delta_1}\right).
%\frac{ \alpha_4^2 \omega_4}{4} \lambda \delta_2^2 \log\left(\frac{\delta_1}{\delta_2}\right) + O\left( \lambda \delta_2^2\right) \leq 2\alpha_4^2\omega_4\left(\frac{\delta_2}{\delta_1}\right) +O\left(\left(\frac{\delta_2}{\delta_1}\right)^{\frac{3}{2}}\right) \leq\alpha_4^2 \omega_4 \lambda \delta_2^2 \log\left(\frac{\delta_1}{\delta_2}\right) + O\left( \lambda \delta_2^2\right) . 
\end{equation}
It is clear that  (\ref{eqfinalepohozaev2}) gives a contradiction. In fact, dividing each side of (\ref{eqfinalepohozaev2}) by $\left(\frac{\delta_2}{\delta_1}\right)$ we have
\begin{equation}\label{eqfinalepohozaev3}
%\frac{ \alpha_4^2 \omega_4}{4} \lambda \delta_1\delta_2 \log\left(\frac{\delta_1}{\delta_2}\right) + O\left( \lambda \delta_2^2\right) \leq 2\alpha_4^2\omega_4 +O\left(\left(\frac{\delta_2}{\delta_1}\right)^{\frac{1}{2}}\right) \leq\alpha_4^2 \omega_4 \lambda \delta_1\delta_2 \log\left(\frac{\delta_1}{\delta_2}\right) + O\left( \lambda \delta_2^2\right) . 
O\left(\lambda \delta_1\delta_2 \log\left(\frac{\delta_1}{\delta_2}\right)\right) + o\left(\lambda \delta_1\delta_2 \log\left(\frac{\delta_1}{\delta_2}\right)\right) = 2\alpha_4^2\omega_4 + o\left(1\right).
\end{equation}
%Since $\delta_i$ is of the form (\ref{formadelta}), and, by the assumptions $d_i \rightarrow \bar d_i>0$, then 
Passing to the limit as $\lambda \rightarrow 0$ in (\ref{eqfinalepohozaev3}), taking into account that $\delta_2=o(\delta_1)$, we deduce that $0=2\alpha_4^2\omega_4$ which is a contradiction.
%In fact there are only three possibilities: $1+2\alpha_2=\alpha_2-\alpha_1$, $1+2\alpha_2<\alpha_2-\alpha_1$, $1+2\alpha_2>\alpha_2-\alpha_1$. The first and the second one are algebraically impossible since $\alpha_1$, $\alpha_2$ are positive, the third one contradicts (\ref{eqfinalepohozaev2}). 
\end{proof}

\hfill
\begin{rem}
In \cite{AMP2, AMP3} sign-changing solutions $u_\lambda$ of \eqref{PBN} with low energy were studied, namely solutions such that $$\int_\Omega |\nabla u_\lambda|^2 \ dx \rightarrow 2S^{N/2}.$$
For this kind of solutions it is not difficult to show (see \cite{AMP2}, Theorem 1.1) that there exist two points $a_{1}=a_1(\lambda)$, $a_{2}=a_2(\lambda)$ in $\Omega$ (one of them is the global maximum point of $|u_\lambda|$) and two positive real numbers $\delta_{1}=\delta_{1}(\lambda)$, $\delta_{2}=\delta_{2}(\lambda)$, such that for $N\geq 4$, as $\lambda \rightarrow 0$, we have
$$\|u_\lambda - PU_{\delta_{1},a_{1}} +PU_{\delta_{2},a_{2}}\|  \rightarrow 0, \ \ {\delta_i^{-1}} {d(a_{i},\partial \Omega)}\rightarrow +\infty, \ \hbox{for}\ i=1,2,$$
where $d(a_i,\partial\Omega)$ is the euclidean distance between $a_i$ and the boundary of $\Omega$. Hence these solutions are of the form \eqref{formasoluz} but with possibly different concentration points. In \cite{AMP2}, assuming that the concentration speeds of $u_\lambda^+$ and $u_\lambda^-$ were comparable, it was proved that the positive and the negative part of $u_\lambda$ had to concentrate in two different points.

Since here we assume that the concentration speeds are different, our result also completes the study made in \cite{AMP2}.
\end{rem}

\section{About the estimate on the $C^1$-norm of $w_\lambda$}
 %In this section we estimate the $C^1$-norm of $w_\lambda$ in compact subsets of $\Omega$ containing the blow-up point. The following result holds:
 Here we show that the hypotheses of Theorem \ref{teoprincipintro} on the $C^1$-norm of the remainder term $w_\lambda$ are almost necessary. Indeed we have:
\begin{teo}\label{teostimac1}
Let $\Omega$ be a bounded open set of $\R^N$ with smooth boundary, $N\geq 4$, and let $\xi \in \Omega$.
Let $u_\lambda$ a solution of (\ref{PBN}) of the form $$u_\lambda=PU_{\delta_1,\xi}-PU_{\delta_2,\xi}+w_\lambda,$$ with $\delta_2=o(\delta_1)$ as $\lambda \rightarrow 0$. Assume that the remainder term $w_\lambda $ is uniformly bounded with respect to $\lambda$ in compact subsets of $\Omega$. Then for any open subset $\Omega^{\prime\prime}\subset\subset \Omega$ such that $\xi \in \Omega^{\prime\prime}$ and for all sufficiently small $\epsilon>0$, there exists a positive constant $C=C(\epsilon,N,\Omega^{\prime\prime})$ such that
 $$\|w_\lambda\|_{C^1(\bar\Omega^{\prime\prime})} \leq C \delta_1^{-\frac{N-2}{2}} \delta_2^{-1+O(\epsilon)},$$
 for all sufficiently small $\lambda>0$.
\end{teo}

\begin{proof}
Without loss of generality we assume that $\xi=0$.
By definition $w_\lambda$ satisfies the following:
\begin{equation}\label{eqrestow}
\begin{cases}
-\Delta w_\lambda = \lambda w_\lambda + \lambda (PU_{\delta_1}-PU_{\delta_2})+U_{\delta_2}^p-U_{\delta_1}^p +|u_\lambda|^{2^* - 2}u_\lambda & \hbox{in} \ \Omega\\
w_\lambda=0& \hbox{on} \ \partial\Omega.
\end{cases}
\end{equation}
Let us set $f_\lambda:= \lambda w_\lambda + \lambda (PU_{\delta_1}-PU_{\delta_2})+U_{\delta_2}^p-U_{\delta_1}^p +|u_\lambda|^{2^* - 2}u_\lambda$. Since $w_\lambda$ and $u_\lambda$ are smooth, applying the Calder\'on-Zygmund inequality we deduce that for any $p \in (1,\infty)$, for any $\Omega^{\prime\prime} \subset \subset \Omega^{\prime} \subset \subset \Omega$ it holds:
\begin{equation}\label{CaldZigm}
\|w_\lambda\|_{2,p,\Omega^{\prime\prime}} \leq C(|w_\lambda|_{p,\Omega^\prime}+|f_\lambda|_{p,\Omega^\prime}), 
\end{equation}
where $C$ depends on $\Omega^\prime$, $N$, $p$, $\Omega^{\prime\prime}$. Thanks to the Sobolev imbedding theorem, for any $\epsilon>0$, if $p=N+\epsilon$ we have that $W^{2,p}(\Omega)$ is continuously imbedded in $C^{1,\gamma}(\bar\Omega)$, where $\gamma=1-\frac{N}{N+\epsilon}$. %Thus, 
Let us consider two open subsets $\Omega^{\prime\prime}$, $\Omega^{\prime}$ of $\Omega$ such that $0 \in \Omega^{\prime\prime}$ and $\Omega^{\prime\prime} \subset \subset \Omega^{\prime} \subset \subset \Omega$. Thanks to \eqref{eqrestow} and \eqref{CaldZigm}, in order to estimate $\|w_\lambda\|_{C^1(\bar\Omega^{\prime\prime})}$ we have to estimate 
the following quantities: $|w_\lambda|_{N+\epsilon,\Omega^\prime}$, $|f_\lambda|_{N+\epsilon,\Omega^\prime}$. %$ |\lambda w_\lambda + \lambda (PU_{\delta_1}-PU_{\delta_2})+U_{\delta_2}^p-U_{\delta_1}^p +|u_\lambda|^{2^* - 2}u_\lambda|_{N+\epsilon,\Omega^\prime}$.

Thanks to the assumptions on $w_\lambda$ we deduce immediately that $|w_\lambda|_{N+\epsilon,\Omega^\prime}=O(1)$, uniformly with respect to $\lambda$. For the other term we argue as it follows: we set $g(s):=|s|^{2^*-2}s$, $\Phi_\lambda:=w_\lambda+\varphi_2-\varphi_1$, where $\varphi_j:=U_{\delta_j}-PU_{\delta_j}$, for $j=1,2$, and we write
\begin{equation*}
\begin{array}{lll}
 &&\displaystyle|f_\lambda|_{N+\epsilon,\Omega^\prime}\\[16pt]
%&& \displaystyle|\lambda w_\lambda + \lambda (PU_{\delta_1}-PU_{\delta_2})+U_{\delta_2}^p-U_{\delta_1}^p +g(u_\lambda)|_{N+\epsilon,\Omega^\prime}\\[16pt]
&\leq& \displaystyle\lambda |w_\lambda|_{N+\epsilon,\Omega^\prime} + \lambda |PU_{\delta_1}|_{N+\epsilon,\Omega^\prime} + \lambda |PU_{\delta_2}|_{N+\epsilon,\Omega^\prime}+ |U_{\delta_1}^p|_{N+\epsilon,\Omega^\prime}\\[16pt] 
&& \displaystyle+ |g(U_{\delta_1}-U_{\delta_2}+\Phi_\lambda)-g(-U_{\delta_2})|_{N+\epsilon,\Omega^\prime}\\[16pt]
&\leq& \displaystyle\lambda |w_\lambda|_{N+\epsilon,\Omega^\prime} + \lambda |PU_{\delta_1}|_{N+\epsilon,\Omega^\prime} + \lambda |PU_{\delta_2}|_{N+\epsilon,\Omega^\prime}+ |U_{\delta_1}^p|_{N+\epsilon,\Omega^\prime}\\[16pt] 
&& \displaystyle+ |g(U_{\delta_1}-U_{\delta_2}+\Phi_\lambda)-g(-U_{\delta_2})-g^\prime(-U_{\delta_2})(U_{\delta_1}+\Phi_\lambda)|_{N+\epsilon,\Omega^\prime} + |g^\prime(-U_{\delta_2})(U_{\delta_1}+\Phi_\lambda)|_{N+\epsilon,\Omega^\prime}\\[16pt]
&=&A+B+C+D+E+F.
\end{array}
\end{equation*}

The term $A$ has been estimated before, and hence $\lambda|w_\lambda|_{N+\epsilon,\Omega^\prime}=O(\lambda)$. For $B$ and $C$ we use the following estimates:
\begin{equation*}
\begin{array}{lllll}
&&\displaystyle \int_{\Omega^{\prime}}\alpha_N^{N+\epsilon} \frac{\delta_j^{\frac{N-2}{2}(N+\epsilon)}}{(\delta_j^2+|x|^2)^{\frac{N-2}{2}(N+\epsilon)}} \ dx = \alpha_N^{N+\epsilon} \int_{\Omega^{\prime}/\delta_j} \frac{\delta_j^{-\frac{N-2}{2}(N+\epsilon)+N}}{(1+|y|^2)^{\frac{N-2}{2}(N+\epsilon)}} \ dy\\[16pt]
&=&\displaystyle \alpha_N^{N+\epsilon} \delta_j^{\frac{4-N}{2}N-\epsilon \frac{N-2}{2}} \int_{\R^N} \frac{1}{(1+|y|^2)^{\frac{N-2}{2}(N+\epsilon)}} \ dy \displaystyle\\[16pt] 
&&+ \displaystyle O\left(\delta_j^{\frac{4-N}{2}N-\epsilon \frac{N-2}{2}}\int_{1/\delta_j}^{+\infty} \frac{r^{N-1}}{(1+r^2)^{\frac{N-2}{2}(N+\epsilon)}} \ dr \right).
\end{array}
\end{equation*}
Thus, for all $\epsilon>0$ sufficiently small we have
\begin{equation*}
\begin{array}{lllll}
|PU_{\delta}|_{N+\epsilon,\Omega^\prime}&\leq&\displaystyle \left(\int_{\Omega^{\prime}}\alpha_N^{N+\epsilon} \frac{\delta_j^{\frac{N-2}{2}(N+\epsilon)}}{(\delta_j^2+|x|^2)^{\frac{N-2}{2}(N+\epsilon)}} \ dx\right)^{\frac{1}{N+\epsilon}}\\[16pt]
&=&\displaystyle \alpha_N \delta_j^{\frac{4-N}{2}+O(\epsilon)} \left(\int_{\R^N} \frac{1}{(1+|y|^2)^{\frac{N-2}{2}(N+\epsilon)}} \ dy\right)^{\frac{1}{N+\epsilon}} + \displaystyle o\left(\delta_j^{\frac{4-N}{2}+ O(\epsilon)}\right).
\end{array}
\end{equation*}
From this we deduce that $B=O( \lambda \delta_1^{\frac{4-N}{2}+O(\epsilon)})$, $C=O(\lambda \delta_2^{\frac{4-N}{2}+O(\epsilon)})$. Concerning the term $D$, with similar computations we see that
\begin{equation*}
\begin{array}{lllll}
|PU_{\delta_1}^p|_{N+\epsilon,\Omega^\prime}&\leq&\displaystyle \left(\int_{\Omega^{\prime}}\alpha_N^{\frac{N+2}{2}(N+\epsilon)} \frac{\delta_1^{\frac{N+2}{2}(N+\epsilon)}}{(\delta_1^2+|x|^2)^{\frac{N+2}{2}(N+\epsilon)}} \ dx\right)^{\frac{1}{N+\epsilon}}\\[16pt]
&=&\displaystyle \alpha_N^p \delta_1^{-\frac{N}{2}+O(\epsilon)} \left(\int_{\R^N} \frac{1}{(1+|y|^2)^{\frac{N+2}{2}(N+\epsilon)}} \ dy\right)^{\frac{1}{N+\epsilon}} + \displaystyle o\left(\delta_1^{-\frac{N}{2}+ O(\epsilon)}\right),
\end{array}
\end{equation*}
and hence $D = O(\delta_1^{-\frac{N}{2}+O(\epsilon)})$.
In order to estimate $E$ we remember that by elementary inequalities we have $|g(u+v)-g(u)-g^\prime(u)v|\leq c|v|^p $, for all $u,v \in \R$, for some constant depending only on $p$, and hence we get that 
$$E\leq c ||\Phi_\lambda|^p |_{N+\epsilon,\Omega^\prime} = O(1).$$
For the last term we have the following:

\begin{equation*}
\begin{array}{lllll}
|g^\prime(U_{\delta_2})U_{\delta_1}|_{N+\epsilon,\Omega^\prime}^{N+\epsilon}&=&\displaystyle p^{N+\epsilon} \int_{\Omega^{\prime}}\alpha_N^{\frac{N+2}{2}(N+\epsilon)} \frac{\delta_2^{\frac{4}{N-2}\frac{N-2}{2}(N+\epsilon)}}{(\delta_2^2+|x|^2)^{\frac{4}{N-2}\frac{N-2}{2}(N+\epsilon)}} \frac{\delta_1^{\frac{N-2}{2}(N+\epsilon)}}{(\delta_1^2+|x|^2)^{\frac{N-2}{2}(N+\epsilon)}}\ dx\\[16pt]
&=&\displaystyle p^{N+\epsilon}\alpha_N^{\frac{N+2}{2}(N+\epsilon)} \int_{\Omega^{\prime}} \frac{\delta_2^{-2(N+\epsilon)}}{(1+|x/\delta_2|^2)^{2(N+\epsilon)}} \frac{\delta_1^{-\frac{N-2}{2}(N+\epsilon)}}{(1+|x/\delta_1|^2)^{\frac{N-2}{2}(N+\epsilon)}}\ dx\\[16pt]
&\leq&\displaystyle p^{N+\epsilon}\alpha_N^{\frac{N+2}{2}(N+\epsilon)} \delta_1^{-\frac{N-2}{2}(N+\epsilon)} \delta_2^{-2(N+\epsilon)+N} \int_{\Omega^{\prime}/\delta_2} \frac{1}{(1+|x/\delta_2|^2)^{2(N+\epsilon)}}  dy\\[16pt]
&\leq&\displaystyle p^{N+\epsilon}\alpha_N^{\frac{N+2}{2}(N+\epsilon)} \delta_1^{-\frac{N-2}{2}(N+\epsilon)} \delta_2^{-N-2\epsilon} \int_{\Omega^{\prime}/\delta_2} \frac{1}{(1+|y|^2)^{2(N+\epsilon)}}  dy\\[16pt]
&=&\displaystyle p^{N+\epsilon}\alpha_N^{\frac{N+2}{2}(N+\epsilon)} \delta_1^{-\frac{N-2}{2}(N+\epsilon)} \delta_2^{-N-2\epsilon} \int_{\R^N} \frac{1}{(1+|y|^2)^{2(N+\epsilon)}}  dy\\[16pt]
&&\displaystyle + O\left(\delta_1^{-\frac{N-2}{2}(N+\epsilon)} \delta_2^{-N-2\epsilon} \int_{1/\delta_2}^{+\infty} \frac{r^{N-1}}{(1+r^2)^{2(N+\epsilon)}}\right).\\[16pt]
%&\leq&\displaystyle p^{N+\epsilon}\alpha_N^{\frac{N+2}{2}(N+\epsilon)} \delta_1^{-\frac{N-2}{2}(N+\epsilon)}\delta_2^{-2(N+\epsilon)+N} \int_{\Omega^{\prime}/\delta_2} \frac{1}{(1+|y|^2)^{\frac{4}{N-2}(N+\epsilon)}} \ dx\\[16pt]
\end{array}
\end{equation*}
Hence we get that
\begin{equation*}
\begin{array}{lllll}
|g^\prime(U_{\delta_2})U_{\delta_1}|_{N+\epsilon,\Omega^\prime}
&\leq&\displaystyle p \alpha_N^{\frac{N+2}{2}} \delta_1^{-\frac{N-2}{2}} \delta_2^{-1+O(\epsilon)} \left(\int_{\R^N} \frac{1}{(1+|y|^2)^{2(N+\epsilon)}}  dy\right)^{\frac{1}{N+\epsilon}} + o\left(\delta_1^{-\frac{N-2}{2}} \delta_2^{-1+O(\epsilon)}\right).
%&\leq&\displaystyle p^{N+\epsilon}\alpha_N^{\frac{N+2}{2}(N+\epsilon)} \delta_1^{-\frac{N-2}{2}(N+\epsilon)}\delta_2^{-2(N+\epsilon)+N} \int_{\Omega^{\prime}/\delta_2} \frac{1}{(1+|y|^2)^{\frac{4}{N-2}(N+\epsilon)}} \ dx\\[16pt]
\end{array}
\end{equation*}
By the same computations we see that
\begin{equation*}
\begin{array}{lllll}
|g^\prime(U_{\delta_2})\Phi_\lambda|_{N+\epsilon,\Omega^\prime}
&=&\displaystyle O\left( \delta_2^{-1+O(\epsilon)} \right).
%&\leq&\displaystyle p^{N+\epsilon}\alpha_N^{\frac{N+2}{2}(N+\epsilon)} \delta_1^{-\frac{N-2}{2}(N+\epsilon)}\delta_2^{-2(N+\epsilon)+N} \int_{\Omega^{\prime}/\delta_2} \frac{1}{(1+|y|^2)^{\frac{4}{N-2}(N+\epsilon)}} \ dx\\[16pt]
\end{array}
\end{equation*}
Thus, we get that $$ |F|\leq c(N,p) \delta_1^{-\frac{N-2}{2}} \delta_2^{-1+O(\epsilon)}.$$
Summing up all these estimates, from (\ref{CaldZigm}) and Sobolev imbedding theorem we deduce that $$\|w_\lambda\|_{C^1(\bar\Omega^{\prime\prime})} \leq C \delta_1^{-\frac{N-2}{2}} \delta_2^{-1+O(\epsilon)}, $$
where $C$ is a positive constant depending on $\epsilon,N,\Omega^{\prime\prime}, \Omega^{\prime}$.
\end{proof}
A straightforward consequence of the previous theorem is the following result:
\begin{cor}\label{cor1}
Under the assumptions of Theorem \ref{teostimac1}, for all sufficiently small $\epsilon>0$ we have
$$\int_{\partial B_r} |\nabla w_\lambda|^2 |x| \ d\sigma \leq C(\epsilon,N) \left(\frac{\delta_2}{\delta_1}\right)^{\frac{N-4}{2}}\delta_2^{O(\epsilon)}, $$
for all sufficiently small $\lambda>0$, where $B_r$ is the ball centered at $\xi$ having radius $r=\sqrt{\delta_1\delta_2}$.
\end{cor}

\section{Concentration speeds for $N\geq 7$}
We consider as in the previous sections sign-changing solutions of Problem \ref{PBN} which are of the form $u_\lambda=PU_{\delta_1,\xi}-PU_{\delta_2,\xi}+w_\lambda$, with $\delta_1=\delta_1(\lambda)$, $\delta_2=\delta_2(\lambda)$ satisfying $\delta_2=o(\delta_1)$ as $\lambda \rightarrow 0$. %Following the ideas contained in \cite{Rey} we generalize them in order to determine the speed of $\delta_1$, under the assumption that 
In addition we assume that $\delta_i$, for $i=1,2$, is of the form 

\begin{equation}\label{formadelta}
\delta_i=d_i \lambda^{\alpha_i},
\end{equation}
%$\alpha_i>0$ is a positive real number, for $i=1,2$.
%\begin{equation}
%\delta_i=d_i \lambda^{\alpha_i},
%\end{equation}
%$\delta_i=d_i \lambda^{\alpha_i}$, 
where $d_i=d_i(\lambda)$ is a strictly positive function such that $d_i \rightarrow \bar d_i>0$, as $\lambda\rightarrow 0$, and the exponents $\alpha_i$ satisfy $0<\alpha_1< \alpha_2$.
%it is of the form (\ref{formadelta}). 
%Moreover, 
Following the ideas contained in \cite{Rey} and applying the asymptotic relation \eqref{eqfinalepohozaev}, found in the proof of Theorem \ref{teoprincipintro}, we determine precisely the exponents $\alpha_1$, $\alpha_2$ in the case $N\geq 7$. We observe that these speeds are exactly the same used in \cite{IACVAIR} to construct solutions of \eqref{PBN} of the form \eqref{formasoluz}.
%Differently from the previous sections, we don't require any extra assumptions on the remainder term $w_\lambda$ apart from the natural orthogonality conditions to the subspace spanned by $PU_{\delta_i}$, $P\frac{\partial U_{\delta_i}}{\partial \delta_i}$, for $i=1,2$.

\begin{teo}\label{teo2}
Let $\Omega$ be a bounded open set of $\R^N$ with smooth boundary, $N\geq 7$, and let $\xi \in \Omega$. 
Let $u_\lambda$ a solution of (\ref{PBN}) such that $u_\lambda$ is of the form $u_\lambda=PU_{\delta_1,\xi}-PU_{\delta_2,\xi}+w_\lambda$, where $\delta_i$, for $i=1,2$, is of the form (\ref{formadelta}) with $\alpha_2 > \alpha_1>0$, $w_\lambda \in V_{\lambda,\xi}$, $V_{\lambda,\xi}$ is the subspace of $H_0^1(\Omega)$:
$$V_{\lambda,\xi}:=\left\{v \in H_0^1(\Omega);\ \ (v,PU_{\delta_i,\xi})_{H_0^1(\Omega)}=\left(v,P\frac{\partial U_{\delta_i,\xi}}{\partial \delta_i}\right)_{H_0^1(\Omega)}=0, \ \ i=1,2\right\}.$$
Moreover assume that  $|w_\lambda|=o(\delta_1^{-\frac{N-2}{2}})$, $|\nabla w_\lambda|=o( \delta_1^{-\frac{N}{2}})$, uniformly in compact subsets of $\Omega$. Then $\alpha_1=\frac{1}{N-4}$, $\alpha_2=\frac{3N-10}{(N-4)(N-6)}$.
\end{teo}

In order to prove Theorem \ref{teo2} we need some preliminary lemmas. Without loss of generality we assume that $\xi=0$. The first one is the following:
\begin{lem}\label{lem1}
Let $\Omega$ be a bounded open set of $\R^N$ with smooth boundary and assume that $0 \in \Omega$, $N\geq5$. Then, as $\delta \rightarrow 0$, we have
$$\int_{\partial \Omega} \left(\frac{\partial PU_{\delta}}{\partial \nu}\right)^2 (x \cdot \nu) \ d \sigma = a_2 \delta^{N-2}+ o\left(\delta^{N-2}\right),$$
for some positive real number $a_2$, depending only on $N$ and $\Omega$. 
\end{lem}
\begin{proof}\
We multiply the equation $-\Delta P U_{\delta}=U_{\delta}^p$ by $\sum_{i=1}^N x_i \frac{\partial P U_{\delta}}{\partial x_i}$ and we integrate on $\Omega$. On one hand, integrating by parts we obtain
\begin{equation}\label{eq0lem1}
\begin{array}{lll}
\displaystyle && \displaystyle \int_\Omega - \Delta PU_\delta \sum_{i=1}^N x_i \frac{\partial P  U_{\delta}}{\partial x_i} \ dx\\
&=& \displaystyle \left(1-\frac{N}{2}\right) \int_\Omega |\nabla PU_\delta|^2 \ dx - \frac{1}{2} \int_{\partial \Omega} \left(\frac{\partial PU_{\delta}}{\partial \nu}\right)^2 (x \cdot \nu) \ d \sigma\\[12pt]
&=& \displaystyle \left(1-\frac{N}{2}\right) \int_\Omega U_\delta^p PU_\delta \ dx - \frac{1}{2} \int_{\partial \Omega} \left(\frac{\partial PU_{\delta}}{\partial \nu}\right)^2 (x \cdot \nu) \ d \sigma.\\
\end{array}
\end{equation}
On the other hand, we have
\begin{equation}\label{eq1lem1}
\begin{array}{lll}
 \displaystyle \int_\Omega U_\delta^p \sum_{i=1}^N x_i \frac{\partial P  U_{\delta}}{\partial x_i} \ dx
%&=& \displaystyle \sum_{i=1}^N \int_\Omega (U_\delta^p  x_i) \frac{\partial P  U_{\delta}}{\partial x_i} \ dx\\[12pt]
&=& \displaystyle - \sum_{i=1}^N \int_\Omega \left(U_\delta^p + p x_i U_\delta^{p-1} \frac{\partial   U_{\delta}}{\partial x_i}\right) { P  U_{\delta}} \ dx\\[12pt]
%&=& \displaystyle - \sum_{i=1}^N \int_\Omega U_\delta^p { P  U_{\delta}} \ dx- p  \sum_{i=1}^N \int_\Omega x_i U_\delta^{p-1} \frac{\partial   U_{\delta}}{\partial x_i} {P  U_{\delta}} \ dx\\[12pt]
&=& \displaystyle -N \int_\Omega U_\delta^p { P  U_{\delta}} \ dx- p  \sum_{i=1}^N \int_\Omega x_i U_\delta^{p-1} \frac{\partial   U_{\delta}}{\partial x_i} {P  U_{\delta}} \ dx.
\end{array}
\end{equation}
By elementary computations we see that $$-\sum_{i=1}^N x_i U_\delta^{p-1} \frac{\partial U_{\delta}}{\partial x_i}  = \frac{N-2}{2} U_\delta + \delta \frac{\partial U_\delta}{\partial \delta},$$
and hence from \eqref{eq1lem1} we get that

\begin{equation}\label{eq2lem1}
\begin{array}{lll}
 &&\displaystyle \int_\Omega U_\delta^p \sum_{i=1}^N x_i \frac{\partial P  U_{\delta}}{\partial x_i} \ dx\\[12pt]
&=& \displaystyle -N \int_\Omega U_\delta^p { P  U_{\delta}} \ dx+ p \frac{N-2}{2} \int_\Omega U_\delta^{p} {P  U_{\delta}} \ dx +  p \delta \int_\Omega U_\delta^{p-1} \frac{\partial U_\delta}{\partial \delta} {P  U_{\delta}} \ dx\\[12pt]
%&=& \displaystyle \left(-N+\frac{N+2}{2}\right) \int_\Omega U_\delta^p { P  U_{\delta}} \ dx+  p \delta \int_\Omega U_\delta^{p-1} \frac{\partial U_\delta}{\partial \delta} {P  U_{\delta}} \ dx\\[12pt]
&=& \displaystyle \left(1-\frac{N}{2}\right) \int_\Omega U_\delta^p { P  U_{\delta}} \ dx+  p \delta \int_\Omega U_\delta^{p-1} \frac{\partial U_\delta}{\partial \delta} {P  U_{\delta}} \ dx.\\
\end{array}
\end{equation}

We analyze the last term of \eqref{eq2lem1}. Applying Lemma \ref{svilprbubb} and since it is well known that $$\int_{\R^N}U_\delta^p\frac{\partial U_\delta}{\partial \delta} \ dx=0,$$ we have
\begin{equation}\label{eq3lem1}
\begin{array}{lll}
 \displaystyle   p \delta \int_\Omega U_\delta^{p-1} \frac{\partial U_\delta}{\partial \delta} {P  U_{\delta}} \ dx &= &  \displaystyle   p \delta \int_\Omega U_\delta^{p-1} \frac{\partial U_\delta}{\partial \delta}   U_{\delta} \ dx- p\alpha_N \delta^{\frac{N}{2}} \int_\Omega U_\delta^{p-1} \frac{\partial U_\delta}{\partial \delta} H(x,0)   \ dx \\[12pt]
 && \displaystyle + o\left(\delta^{\frac{N}{2}} \int_\Omega U_\delta^{p-1} \frac{\partial U_\delta}{\partial \delta} H(x,0)   \ dx\right)\\[12pt]
 &= &  \displaystyle   - p \delta \int_{\R^N\setminus\Omega} U_\delta^{p} \frac{\partial U_\delta}{\partial \delta}  \ dx- p\alpha_N \delta^{\frac{N}{2}} \int_\Omega U_\delta^{p-1} \frac{\partial U_\delta}{\partial \delta} H(x,0)   \ dx \\[12pt]
 && \displaystyle + o\left(\delta^{\frac{N}{2}} \int_\Omega U_\delta^{p-1} \frac{\partial U_\delta}{\partial \delta} H(x,0)   \ dx\right),
\end{array}
\end{equation}
where $H$ denotes, the regular part of the Green function for the Laplacian. By definition it is easy to see that
\begin{equation}\label{eq4lem1}
\begin{array}{lll}
 \displaystyle  \left|-p \delta \int_{\R^N\setminus\Omega} U_\delta^{p} \frac{\partial U_\delta}{\partial \delta}  \ dx\right| &\leq& \displaystyle   \alpha_N^{p+1} \frac{N+2}{2} \ \delta \int_{\R^N\setminus\Omega} \frac{\delta^{\frac{N+2}{2}}}{\left(\delta^2+|x|^2\right)^{\frac{N+2}{2}}} \frac{\delta^{\frac{N-2}{2}}\left||x|^2-\delta^2\right|}{\left(\delta^2+|x|^2\right)^{\frac{N}{2}}}\ dx\\[16pt]
% &\leq& \displaystyle   \alpha_N^{p+1} \frac{N+2}{2} \ \int_{\R^N\setminus\Omega} \frac{\delta^{N+1}}{\left(\delta^2+|x|^2\right)^{\frac{N+2}{2}}} \frac{\left||x|^2-\delta^2\right|}{\left(\delta^2+|x|^2\right)^{\frac{N}{2}}}\ dx\\[16pt]
  &\leq& \displaystyle   \alpha_N^{p+1} \frac{N+2}{2} \ \int_{\R^N\setminus\Omega} \frac{\delta^{N+1}}{|x|^{N+2}} \frac{\left||x|^2-\delta^2\right|}{|x|^{{N}}}\ dx\\[16pt]
 %&\leq& \displaystyle   \alpha_N^{p+1} \frac{N+2}{2} \left( \int_{\R^N\setminus\Omega} \frac{\delta^{N+1}}{|x|^{N+2}} \frac{1}{|x|^{{N-2}}}\ dx + \ \int_{\R^N\setminus\Omega} \frac{\delta^{N+1}}{|x|^{N+2}} \frac{\delta^2}{|x|^{{N}}}\ dx\right)\\[16pt]
 &=& \displaystyle   O \left( \delta^{N+1}\right).
 \end{array}
\end{equation}

Moreover, by the usual change of variable and applying the mean value theorem, we have

\begin{equation}\label{eq5lem1}
\begin{array}{lll}
 \displaystyle  p\alpha_N \delta^{\frac{N}{2}} \int_\Omega U_\delta^{p-1} \frac{\partial U_\delta}{\partial \delta} H(x,0)   \ dx &=&  \displaystyle  p\alpha_N^{p+1} \delta^{\frac{N-2}{2}} \int_\Omega \frac{\delta^{2}}{\left(\delta^2+|x|^2\right)^{2}} \frac{\delta^{\frac{N-2}{2}}\left(|x|^2-\delta^2\right)}{\left(\delta^2+|x|^2\right)^{\frac{N}{2}}}  H(x,0)   \ dx\\[16pt]
 &=&  \displaystyle  p\alpha_N^{p+1} \delta^{\frac{N-2}{2}} \int_\Omega \frac{\delta^{2}}{\delta^{4}\left(1+|\frac{x}{\delta}|^2\right)^{2}} \frac{\delta^{\frac{N-2}{2}}\delta^2\left(|\frac{x}{\delta}|^2-1\right)}{\delta^N\left(1+|\frac{x}{\delta}|^2\right)^{\frac{N}{2}}}  H(x,0)   \ dx\\[16pt]
  &=&  \displaystyle  p\alpha_N^{p+1} \delta^{{N-2}} \int_{\Omega/\delta} \frac{1}{\left(1+|y|^2\right)^{2}} \frac{\left(|y|^2-1\right)}{\left(1+|y|^2\right)^{\frac{N}{2}}}  H(\delta y,0)   \ dy\\[16pt]
    &=&  \displaystyle  p\alpha_N^{p+1} \delta^{{N-2}} \int_{\Omega/\delta} \frac{1}{\left(1+|y|^2\right)^{2}} \frac{\left(|y|^2-1\right)}{\left(1+|y|^2\right)^{\frac{N}{2}}}  H(0,0)   \ dy\\[16pt]
    &+&  \displaystyle O\left( \delta^{N-1} \int_{\Omega/\delta} \frac{1}{\left(1+|y|^2\right)^{2}} \frac{\left(|y|^2-1\right)}{\left(1+|y|^2\right)^{\frac{N}{2}}}(\nabla H(\eta y,0)\cdot y)   \ dy\right)\\[16pt]
        &=&  \displaystyle  p\alpha_N^{p+1} \delta^{{N-2}} \int_{\R^N} \frac{1}{\left(1+|y|^2\right)^{2}} \frac{\left(|y|^2-1\right)}{\left(1+|y|^2\right)^{\frac{N}{2}}}  H(0,0)   \ dy\\[16pt]
         &+&  \displaystyle O\left(\delta^{{N-2}} \int_{1/\delta}^{+\infty} \frac{r^{N-1}}{\left(1+r^2\right)^{2}} \frac{\left(r^2-1\right)}{\left(1+r^2\right)^{\frac{N}{2}}}  H(0,0)   \ dr\right)\\[16pt]
    &+&  \displaystyle O\left( \delta^{{N-1}} \int_{\Omega/\delta} \frac{1}{\left(1+|y|^2\right)^{2}} \frac{\left(|y|^2-1\right)}{\left(1+|y|^2\right)^{\frac{N}{2}}}(\nabla H(\eta y,0)\cdot y)   \ dy\right)\\[16pt]
     &=&  \displaystyle  p\alpha_N^{p+1}H(0,0)  \delta^{{N-2}} \int_{\R^N} \frac{\left(|y|^2-1\right)}{\left(1+|y|^2\right)^{\frac{N+4}{2}}}    \ dy + O(\delta^{N-1}).
 \end{array}
\end{equation}
Finally from \eqref{eq0lem1}-\eqref{eq5lem1} we get that
$$ \int_{\partial \Omega} \left(\frac{\partial PU_{\delta}}{\partial \nu}\right)^2 (x \cdot \nu) \ d \sigma=2 p\alpha_N^{p+1}H(0,0)  \delta^{{N-2}} \int_{\R^N} \frac{\left(|y|^2-1\right)}{\left(1+|y|^2\right)^{\frac{N+4}{2}}}    \ dy + O(\delta^{N-1}),$$
and the proof is complete.
\end{proof}
Another preliminary lemma is the following:
\begin{lem}\label{lem2}
Under the assumptions of Theorem \ref{teo2}, as $\lambda \rightarrow 0$, we have
$$\left|\int_{\partial\Omega}\left(\frac{\partial w_\lambda}{\partial \nu}\right)^2 (x \cdot \nu) \ d\sigma \right|= O (\lambda^2\delta_1^{4})+o(\delta_1^{N-2}). $$
\end{lem}
\begin{proof}
The first step is the following:
\begin{eqnarray*}
\left|\int_{\partial\Omega}\left(\frac{\partial w_\lambda}{\partial \nu}\right)^2 (x \cdot \nu) \ d\sigma \right| &\leq&\int_{\partial\Omega}\left(\frac{\partial w_\lambda}{\partial \nu}\right)^2 |x \cdot \nu| \ d\sigma\\
& \leq&\int_{\partial\Omega}\left(\frac{\partial w_\lambda}{\partial \nu}\right)^2 |x| \ d\sigma\\
& \leq& c(\Omega)\int_{\partial\Omega}\left(\frac{\partial w_\lambda}{\partial \nu}\right)^2 \ d\sigma.
\end{eqnarray*}
Thus we need to estimate $\displaystyle\int_{\partial\Omega}\left(\frac{\partial w_\lambda}{\partial \nu}\right)^2 \ d\sigma$. Let us consider a smooth function $\zeta:\R^N \rightarrow \R$ such that $0\leq\zeta\leq1$, $\zeta(x)=0$ for $|x|\leq\frac{1}{2}$ and $\zeta(x)=1$ for $|x|\geq 1$. We set $\eta(x):=\zeta(\frac{x}{d(0,\partial \Omega)})$. It's elementary to see that $\eta w_\lambda$ is a solution of the following problem
\begin{equation}\label{PR}
\begin{cases}
-\Delta (\eta w_\lambda) = \lambda \eta w_\lambda + g_\lambda & \hbox{in}\ \Omega\\
\eta w_\lambda=0 & \hbox{on}\ \partial \Omega,
\end{cases}
\end{equation}
where $g_\lambda=\eta \left(\lambda PU_{\delta_1}-\lambda PU_{\delta_2}- U_{\delta_1}^p+U_{\delta_2}^p+|u_\lambda|^{2^*-2}u_\lambda\right) - 2 \nabla \eta \cdot \nabla w_\lambda - w_\lambda\Delta\eta.$
Since $\eta w_\lambda$ is a solution of \eqref{PR}, the following inequality holds (see Appendix C in \cite{Rey}):
\begin{equation}\label{traceineq}
\left|\frac{\partial}{\partial \nu}\left(\eta w_\lambda\right) \right|_{2,\partial\Omega}^2 = \left|\frac{\partial w_\lambda}{\partial \nu} \right|_{2,\partial\Omega}^2 \leq C |g_\lambda|_{\frac{2N}{N+1},\Omega}^2,
\end{equation}
where $C$ is a positive constant depending only on $\Omega$ and $N$. Hence, in order to complete the proof, it suffices to estimate the $L^{\frac{2N}{N+1}}(\Omega)$-norm of $g_\lambda$. We point out that, thanks to the multiplication by the cut-off function $\eta$, what occurs around the origin does not count anymore and this will make the boundary estimate sharper. By elementary inequalities we get that 
$$|g_\lambda| \leq c(p) \eta \left(\lambda U_{\delta_1}+ \lambda U_{\delta_2}+ U_{\delta_1}^p + U_{\delta_2}^p + |w_\lambda|^p \right) +2 |\nabla \eta| |\nabla w_\lambda| + |\Delta \eta| |w_\lambda|. $$
Thus we have to estimate the following quantities: $$\lambda|\eta U_{\delta_j} |_{\frac{2N}{N+1},\Omega}, \ |\eta U_{\delta_j}^p |_{\frac{2N}{N+1},\Omega},\ \hbox{for} \ j=1,2,\ \hbox{and} \  |\eta |w_\lambda|^p |_{\frac{2N}{N+1},\Omega}, \ |\ |\nabla \eta| |\nabla w_\lambda|\  |_{\frac{2N}{N+1},\Omega},\ |\ |\Delta \eta| |w_\lambda| \ |_{\frac{2N}{N+1},\Omega}.$$
 This is a long computation already made by O. Rey (see Appendix C of \cite{Rey}), in the case of positive solutions of the form $u_\lambda=PU_{\delta}+ w_\lambda$. In that paper it is shown that
 $$|\eta U_{\delta_j}^p|_{\frac{2N}{N+1},\Omega}^2 = o\left( \delta_j^{N-2}\right), \ \ \ |\eta \lambda U_{\delta_j}|_{\frac{2N}{N+1},\Omega}^2 = O\left(\lambda^2 \delta_j^{N-2}\right),$$ 
 \begin{equation}\label{missingterms}
 \ \ \ \Big| |\nabla \eta||\nabla w_\lambda| \Big|_{\frac{2N}{N+1},\Omega}^2 = O\left(\|w_\lambda\|^2\right),\ \ \ \Big| |\Delta \eta||w_\lambda|\Big|_{\frac{2N}{N+1},\Omega}^2 = O\left(\|w_\lambda\|^2\right).
 \end{equation}
 
Moreover, by the same computations of Appendix C in \cite{Rey} we see that 
  $$\Big| \eta|w_\lambda|^p\Big|_{\frac{2N}{N+1},\Omega}^2=o(\delta_1^{N-2}).$$
    %It is clear that if we show that $\|w_\lambda\|=o(\delta_1)$ the proof is complete.
   In order to complete the proof  we need to estimate the quantities in (\ref{missingterms}), and hence we have to study the asymptotic behavior of $\|w_\lambda\|$. An estimate for $\|w_\lambda\|$ is contained in \cite{AMP2}; in particular, by the proof of Lemma 3.3 of \cite{AMP2} we see that
    \begin{equation} \label{stimarestopacella}
   \|w_\lambda\| \leq c \left[\sum_i \left(\lambda \delta_i^{(N-2)/2} + \delta_i^{N-2}\right)+\epsilon_{12}(\log \epsilon_{12}^{-1})^{(N-2)/N}\right], 
    \end{equation}
    where $\epsilon_{12}$ is defined by $\epsilon_{12}:=\left(\frac{\delta_1}{\delta_2}+\frac{\delta_2}{\delta_1}\right)^{(2-N)/2}.$ Since $\frac{\delta_2}{\delta_1} \rightarrow 0$ as $\lambda \rightarrow 0$ we see that $$\epsilon_{12}=\left(\frac{\delta_2}{\delta_1}\right)^{\frac{N-2}{2}}+ o\left(\frac{\delta_2}{\delta_1}\right)^{\frac{N-2}{2}}.$$
Moreover by the assumptions on the growth of $\nabla w_\lambda$ and $w_\lambda$, and thanks to \eqref{eqfinalepohozaev} we get that $\epsilon_{12} $ is of the same order as $\lambda \delta_2^2$, hence, since $\delta_2=o(\delta_1)$ as $\lambda \rightarrow 0$, we have that $$\epsilon_{12}(\log \epsilon_{12}^{-1})^{(N-2)/N}=o(\lambda\delta_1^2).$$
Thus, from \eqref{stimarestopacella}, and since $N\geq7$, we deduce that for all sufficiently small $\lambda$ it holds
\begin{equation}\label{stimafinresto}
\|w_\lambda\|\leq c(\delta_1^{N-2}+\lambda\delta_1^2).
 \end{equation}
Summing up all these estimates we deduce the desired relation.
 \end{proof}
 
 \begin{lem}\label{lem3}
Let $\Omega$ be a bounded open set of $\R^N$ with smooth boundary and assume that $0 \in \Omega$, $N\geq5$. Then, as $\delta \rightarrow 0$, we have
$$\int_{\partial \Omega} \left(\frac{\partial PU_{\delta}}{\partial \nu}\right)^2 \ d \sigma = O (\delta^{N-2}).$$
\end{lem}
\begin{proof}
We consider a smooth function $\eta:\R^N \rightarrow \R$ having the same properties as the one considered in the previous proof. By elementary computation we see that $\eta  PU_{\delta}$ satisfies
\begin{equation}\label{PR2}
\begin{cases}
-\Delta (\eta PU_{\delta}) = - (\Delta\eta) PU_{\delta} -\nabla \eta \cdot \nabla PU_{\delta} + \eta U_{\delta}^p & \hbox{in}\ \Omega\\
 \eta PU_{\delta}=0 & \hbox{on}\ \partial \Omega.
\end{cases}
\end{equation}
Since $\eta PU_{\delta}$ is a solution of \eqref{PR2}, the following inequality holds:
\begin{equation}\label{lem3eq1}
\left|\frac{\partial}{\partial \nu}\left(\eta PU_{\delta}\right) \right|_{2,\partial\Omega}^2 = \left|\frac{\partial PU_{\delta}}{\partial \nu} \right|_{2,\partial\Omega}^2 \leq C\Big| |\Delta\eta| PU_{\delta} +|\nabla \eta \cdot \nabla PU_{\delta}| + \eta U_{\delta}^p\Big|_{\frac{2N}{N+1},\Omega}^2,
\end{equation}
where $C$ is a positive constant depending only on $\Omega$ and $N$. In order to complete the proof we have to estimate the quantities: $|(\Delta\eta) PU_{\delta}|_{\frac{2N}{N+1}^2,\Omega}$, $|\nabla \eta \cdot \nabla PU_{\delta}|_{\frac{2N}{N+1},\Omega}^2$, $| \eta U_{\delta}^p|_{\frac{2N}{N+1},\Omega}^2$. 
Using the same computations made by O. Rey in  \cite{Rey}, and since  $\eta\equiv 0$ in a neighborhood of the origin we get that
\begin{equation}
\begin{array}{lllll}\label{lem3eq2}
 \displaystyle |\eta U_{\delta}^p|_{\frac{2N}{N+1},\Omega}^2 =  \displaystyle o\left( \delta^{N-2}\right), \ \ \ \displaystyle \Big||\nabla \eta||\nabla PU_\delta|\Big|_{\frac{2N}{N+1},\Omega}^2 =  \displaystyle O\left(\|PU_\delta\|_{\Omega\cap supp( \nabla\eta)}^2\right), \\[12pt]
  \displaystyle \Big| |\Delta \eta|| PU_\delta|\Big|_{\frac{2N}{N+1},\Omega}^2 =  \displaystyle O\left(\|PU_\delta\|_{\Omega\cap supp(\nabla \eta)}^2\right).
 \end{array}
 \end{equation}
Applying Lemma \ref{svilprbubb} and taking account of (\ref{normal2udelta}), since $\nabla \eta \equiv 0$ in an open neighborhood of the origin, we have
\begin{equation}\label{lem3eq3}
\begin{array}{lllll}
\displaystyle \|PU_\delta\|_{\Omega\cap supp( \nabla\eta)}^2 &=&\displaystyle \int_{\Omega\cap supp(\nabla \eta)} |\nabla (U_\delta-\varphi_\delta)|^2 \ dx \\[16pt ]&\leq &\displaystyle \int_{\Omega\cap supp( \nabla\eta)} |\nabla U_\delta|^2 dx  + 2\int_{\Omega\cap supp( \nabla\eta)} |\nabla U_\delta||\nabla \varphi_\delta| dx\\[16pt] 
&&\displaystyle + \int_{\Omega\cap supp( \nabla\eta)} |\nabla \varphi_\delta|^2 dx\\[16pt]
&=&\displaystyle O(\delta^{N-2}).
%&\leq&\displaystyle \int_{\Omega\cap supp( \eta)} U_\delta^{p+1} \ dx\\
 % &\leq& \displaystyle \alpha_N^{p+1} \frac{\delta^N}{(\delta^2+d(0,supp(\eta))} \ |\Omega\cap supp( \eta)| \\[16pt]
%&=& \displaystyle c(\eta,\Omega,N) \delta^N .
\end{array}
\end{equation}
From \eqref{lem3eq1}, \eqref{lem3eq2} and \eqref{lem3eq3} we deduce that $$ \left|\frac{\partial PU_{\delta}}{\partial \nu} \right|_{2,\partial\Omega}^2 =O(\delta^{N-2}), $$
and the proof is complete.
\end{proof}
 
\begin{proof}[\bf{Proof of Theorem \ref{teo2}}]
We apply the Pohozaev's identity to $u_\lambda=PU_{\delta_1}-PU_{\delta_2}+w_\lambda$. Since $u_\lambda$ is a solution of Problem \ref{PBN} we have

\begin{equation}\label{stdpohozaev}
\lambda \int_{\Omega} u_\lambda^2 \ dx = \frac{1}{2} \int_{\partial \Omega }\left(\frac{\partial u_\lambda}{\partial \nu}\right)^2 (x\cdot\nu)  \ d\sigma.
\end{equation}

For the left-hand side of \eqref{stdpohozaev}, as in the previous proofs we set $\Phi_\lambda:=w_\lambda-\varphi_{\delta_1}+\varphi_{\delta_2}$, where $\varphi_{\delta_j}=U_{\delta_j}-PU_{\delta_j}$ for $j=1,2$, and we have
\begin{equation} \label{eq1teo2}
\begin{array}{lll}
\displaystyle \lambda \int_{\Omega} u_\lambda^2 \ dx&=&\displaystyle \lambda \int_{\Omega} (PU_{\delta_1}-PU_{\delta_2}+w_\lambda)^2 \ dx\\[12pt]
%&=&\displaystyle\lambda \int_{B_{r}} (U_{\delta_1}-U_{\delta_2}-\varphi_{\delta_1}+\varphi_{\delta_2}+w_\lambda)^2 \ dx\\[12pt]
&=&\displaystyle\lambda \int_{\Omega} (U_{\delta_1}-U_{\delta_2}+\Phi_\lambda)^2 \ dx\\[12pt]
&=&\displaystyle\lambda \int_{\Omega}\left(U_{\delta_1}^2+U_{\delta_2}^2 - 2U_{\delta_1}U_{\delta_2}+2U_{\delta_1}\Phi_\lambda-2U_{\delta_2}\Phi_\lambda+\Phi_\lambda^2\right) \ dx\\[12pt]
&=&A+B+C+D+E+F.
\end{array}
\end{equation}

In order to estimate $A$ and $B$ we use the following

\begin{eqnarray}\label{eq2teo2}
\begin{array}{lll}
\displaystyle {\lambda} \int_\Omega U_{\delta_j}^2 \ dx &=&\displaystyle\lambda \ \alpha_N^2 \int_\Omega \frac{\delta_j^{-(N-2)}}{(1+|x/\delta_j|^2)^{N-2}} \ dx = \displaystyle\lambda \ \alpha_N^2 \int_{\Omega/\delta_j} \frac{\delta_j^{-(N-2)}}{(1+|y|^2)^{N-2}} \delta_j^N \ dy\\[12pt]
  &=&\displaystyle\lambda \ \alpha_N^2  \delta_j^2 \int_{\R^N} \frac{1}{(1+|y|^2)^{N-2}} \ dy  +\displaystyle O\left( \lambda\delta_j^2 \int_{1/\delta_j}^{+\infty} \frac{r^{N-1}}{(1+r^2)^{N-2}} \ dr \right) \\[12pt]
  &=&\displaystyle\lambda \ \alpha_N^2  \delta_j^2 \int_{\R^N} \frac{1}{(1+|y|^2)^{N-2}} \ dy + O\left( \lambda\delta_j^{N-2} \right).
   \end{array}
\end{eqnarray}

We point out that since we are assuming that $N\geq 5$, the first integral in the last line of \eqref{eq2teo2} converges.
To estimate $C$ we apply the following
\begin{eqnarray}\label{eq3teo2}
\begin{array}{lll}
\displaystyle \lambda \int_\Omega U_{\delta_1}U_{\delta_2} \ dx
&=&\displaystyle\lambda \ \alpha_N^2 \int_{\Omega/\delta_1} \frac{\delta_1^{\frac{N+2}{2}}}{(1+|y|^2)^{\frac{N-2}{2}}}  \frac{\delta_2^{\frac{N-2}{2}}}{(\delta_2^2+\delta_1^2|y|^2)^{\frac{N-2}{2}}} \ dy\\[12pt]
&=&\displaystyle\lambda \ \alpha_N^2 \int_{\Omega/\delta_1} \frac{\delta_1^{-\frac{N-6}{2}}}{(1+|y|^2)^{\frac{N-2}{2}}}  \frac{\delta_2^{\frac{N-2}{2}}}{\left(\left(\frac{\delta_2}{\delta_1}\right)^2+|y|^2\right)^{\frac{N-2}{2}}} \ dy\\[12pt]
&\leq&\displaystyle\lambda \ \alpha_N^2 \left(\frac{\delta_2}{\delta_1}\right)^{\frac{N-2}{2}} \delta_1^2 \int_{\Omega/\delta_1} \frac{1}{(1+|y|^2)^{\frac{N-2}{2}}|y|^{N-2}} \ dy \\[12pt]
&=&\displaystyle\lambda \ \alpha_N^2 \left(\frac{\delta_2}{\delta_1}\right)^{\frac{N-2}{2}} \delta_1^2 \int_{\R^N} \frac{1}{(1+|y|^2)^{\frac{N-2}{2}}|y|^{N-2}} \ dy\\[12pt]
&+& \displaystyle O\left(\lambda \left(\frac{\delta_2}{\delta_1}\right)^{\frac{N-2}{2}} \delta_1^2 \int_{1/\delta_1}^{+\infty} \frac{r^{N-1}}{(1+r^2)^{\frac{N-2}{2}}r^{N-2}} \ dr\right)\\[12pt]
&=&\displaystyle\lambda \ \alpha_N^2 \left(\frac{\delta_2}{\delta_1}\right)^{\frac{N-2}{2}} \delta_1^2 \int_{\R^N} \frac{1}{(1+|y|^2)^{\frac{N-2}{2}}|y|^{N-2}} \ dy+ \displaystyle O\left(\lambda \left(\frac{\delta_2}{\delta_1}\right)^{\frac{N-2}{2}} \delta_1^{N-2}\right).
\end{array}
\end{eqnarray}
In order to estimate $D$, $E$, $F$, thanks to \eqref{stimafinresto}, H\"older's inequality and Poincar\'e's inequality we get that
\begin{equation}\label{eq4teo2}
\int_\Omega w_\lambda^2 \leq c_1 \|w_\lambda\|^2 \leq c_2 (\delta_1^{N-2}+\lambda\delta_1^2)^2 . 
\end{equation}
We observe that, by  Lemma \ref{svilprbubb} and since $N\geq5$, we have $|\varphi_{\delta_j}|_{2,\Omega}=O\left(\delta_j^{\frac{N-2}{2}}\right)=o(\delta_j)$. Thus, by definition of $\Phi_\lambda$ and \eqref{eq4teo2} we deduce that
\begin{equation}\label{eq5teo2}
 \int_\Omega \Phi_\lambda^2 \ dx =  \int_\Omega \left(w_\lambda + \varphi_{\delta_2} -  \varphi_{\delta_1} \right)^2 \ dx = o(\delta_1^2),
\end{equation}
and hence
\begin{equation}\label{eq6teo2}
F= o(\lambda\delta_1^2).
\end{equation}
Moreover, by the same computations of \eqref{eq2teo2} we have $\int_\Omega U_{\delta_j}^2 = a_1 \delta_j^2 + o(\delta_j^2)$, for some positive constant $a_1$. Hence by H\"older's inequality and \eqref{eq5teo2} we get that
$$|D|= o(\lambda\delta_1^2), $$
and
$$|E|= o(\lambda\delta_1\delta_2) = o(\lambda\delta_1^2). $$
We analyze now the right-hand side of \eqref{stdpohozaev}: by definition we have
\begin{equation} \label{eq7teo2}
\begin{array}{lll}
\displaystyle  \frac{1}{2} \int_{\partial \Omega }\left(\frac{\partial u_\lambda}{\partial \nu}\right)^2 (x\cdot\nu)  \ d\sigma&=&\displaystyle \frac{1}{2} \int_{\partial \Omega }\left(\frac{\partial PU_{\delta_1}}{\partial \nu}- \frac{\partial PU_{\delta_2}}{\partial \nu} + \frac{\partial w_\lambda}{\partial \nu}\right)^2 (x\cdot\nu)  \ d\sigma\\[16pt]
&=&\displaystyle \frac{1}{2} \int_{\partial \Omega }\left(\frac{\partial PU_{\delta_1}}{\partial \nu}\right)^2 (x\cdot\nu)  \ d\sigma + \frac{1}{2} \int_{\partial \Omega }\left(\frac{\partial PU_{\delta_2}}{\partial \nu}\right)^2 (x\cdot\nu)  \ d\sigma\\[16pt]
&&\displaystyle  -\int_{\partial \Omega } \frac{\partial PU_{\delta_1}}{\partial \nu}\frac{\partial PU_{\delta_2}}{\partial \nu} (x\cdot\nu)  \ d\sigma + \int_{\partial \Omega } \frac{\partial PU_{\delta_1}}{\partial \nu}\frac{\partial w_\lambda}{\partial \nu} (x\cdot\nu)  \ d\sigma\\[16pt]
&&\displaystyle - \int_{\partial \Omega } \frac{\partial PU_{\delta_2}}{\partial \nu}\frac{\partial w_\lambda}{\partial \nu} (x\cdot\nu)  \ d\sigma + \frac{1}{2} \int_{\partial \Omega }\left(\frac{w_\lambda}{\partial \nu}\right)^2 (x\cdot\nu)  \ d\sigma \\[16pt]
&=&A_1+B_1+C_1+D_1+E_1+F_1.
\end{array}
\end{equation}
Thanks to Lemma \ref{lem1} we have:
\begin{equation} \label{eq8teo2}
\begin{array}{lll}
\displaystyle  A_1 &=& \displaystyle  \frac{a}{2} \delta_1^{N-2} + o(\delta_1^{N-2}),\\[12pt]
\displaystyle  B_1 &=& \displaystyle  \frac{a}{2} \delta_2^{N-2} + o(\delta_2^{N-2}).
\end{array}
\end{equation}
%Moreover, applying H\"older inequality we get that
Thanks to Lemma \ref{lem3} and applying H\"older inequality we get that
\begin{equation} \label{eq9teo2}
\begin{array}{lll}
\displaystyle  |C_1| &\leq& \displaystyle \int_{\partial \Omega } \left|\frac{\partial PU_{\delta_1}}{\partial \nu}\right|\left|\frac{\partial PU_{\delta_2}}{\partial \nu}\right| |x\cdot\nu|  \ d\sigma \\[12pt]
&\leq& diam(\partial\Omega) \displaystyle \left(\int_{\partial \Omega } \left|\frac{\partial PU_{\delta_1}}{\partial \nu}\right|^2 \ d\sigma \right)^{\frac{1}{2}}  \left(\int_{\partial \Omega } \left|\frac{\partial PU_{\delta_2}}{\partial \nu}\right|^2 \ d\sigma \right)^{\frac{1}{2}}\\[12pt]
&=&\displaystyle O\left(\delta_1^{\frac{N-2}{2}}\delta_2^{\frac{N-2}{2}}\right).
\end{array}
\end{equation}
Thanks to \eqref{traceineq}, Lemma \ref{lem2}, Lemma \ref{lem3} and applying H\"older inequality we get that
\begin{equation} \label{eq10teo2}
\begin{array}{lll}
\displaystyle  |D_1| &\leq& \displaystyle \int_{\partial \Omega } \left|\frac{\partial PU_{\delta_1}}{\partial \nu}\right|\left|\frac{\partial w_\lambda}{\partial \nu}\right| |x\cdot\nu|  \ d\sigma \\[12pt]
&\leq& diam(\partial\Omega) \displaystyle \left(\int_{\partial \Omega } \left|\frac{\partial PU_{\delta_1}}{\partial \nu}\right|^2 \ d\sigma \right)^{\frac{1}{2}}  \left(\int_{\partial \Omega } \left|\frac{\partial w_\lambda}{\partial \nu}\right|^2 \ d\sigma \right)^{\frac{1}{2}}\\[12pt]
&=&\displaystyle o\left(\lambda \delta_1^{2}\right)+ o\left(\delta_1^{N-2}\right).
\end{array}
\end{equation}
\begin{equation} \label{eq11teo2}
\begin{array}{lll}
\displaystyle  |E_1| &\leq& \displaystyle \int_{\partial \Omega } \left|\frac{\partial PU_{\delta_2}}{\partial \nu}\right|\left|\frac{\partial w_\lambda}{\partial \nu}\right| |x\cdot\nu|  \ d\sigma \\[12pt]
&\leq& diam(\partial\Omega) \displaystyle \left(\int_{\partial \Omega } \left|\frac{\partial PU_{\delta_2}}{\partial \nu}\right|^2 \ d\sigma \right)^{\frac{1}{2}}  \left(\int_{\partial \Omega } \left|\frac{\partial w_\lambda}{\partial \nu}\right|^2 \ d\sigma \right)^{\frac{1}{2}}\\[16pt]
&=&\displaystyle o\left(\lambda \delta_1^{2}\right)+ o\left(\delta_1^{N-2}\right).
\end{array}
\end{equation}
\\
\begin{equation} \label{eq12teo2}
\begin{array}{lllll}
\displaystyle  |F_1| &=&\displaystyle \frac{1}{2} \int_{\partial \Omega }\left(\frac{\partial w_\lambda}{\partial \nu}\right)^2 (x\cdot\nu)  \ d\sigma &=&\displaystyle  o\left(\lambda \delta_1^{2}\right)+ o\left(\delta_1^{N-2}\right).
\end{array}
\end{equation}
Summing up all the estimates, from \eqref{stdpohozaev} and since $\delta_2=o(\delta_1)$ as $\lambda \rightarrow 0$, we deduce the following equality:
\begin{equation}\label{lasteqteo2}
a_1 \lambda \delta_1^2 + o( \lambda \delta_1^2)=a_2 \delta_1^{N-2}+o\left(\delta_1^{N-2}\right).
\end{equation}
Since $\delta_j$ is of the form (\ref{formadelta}), we deduce that $\alpha_1$ must satisfy the equation
$$1+2\alpha_1=(N-2)\alpha_1, $$
and hence we get that
 $\alpha_1=\frac{1}{N-4}$. 
Moreover, from \eqref{eqfinalepohozaev} we deduce that $\alpha_1$, $\alpha_2$ must satisfy the following algebraic equation
\begin{equation}\label{eqalgebrica}
1+2\alpha_2 = \frac{N-2}{2}(\alpha_2 - \alpha_1). 
\end{equation}
 Thus, combining this result with (\ref{eqalgebrica}), we get that $\alpha_2=\frac{3N-10}{(N-4)(N-6)}$ and the proof is complete.  
\end{proof}

\end{document}